\documentclass[11pt]{article}

\usepackage{epsfig,epsf,fancybox}
\usepackage{amsmath}
\usepackage{mathrsfs}
\usepackage{amssymb}
\usepackage{graphicx}
\usepackage{color}
\usepackage{multirow}
\usepackage{paralist}
\usepackage{verbatim}
\usepackage{galois}
\usepackage{algorithm}
\usepackage{algorithmic}
\usepackage{boxedminipage}
\usepackage{booktabs}
\usepackage{accents}
\usepackage{stmaryrd}

\usepackage{subfig}

\usepackage{natbib}

\usepackage{url}
\usepackage[colorlinks,linkcolor=magenta,citecolor=blue, pagebackref=true,backref=true]{hyperref}
\renewcommand*{\backrefalt}[4]{%
    \ifcase #1 \footnotesize{(Not cited.)}%
    \or        \footnotesize{(Cited on page~#2.)}%
    \else      \footnotesize{(Cited on pages~#2.)}%
    \fi}

\newtheorem{theorem}{Theorem}[section]

\newtheorem{lemma}[theorem]{Lemma}
\newtheorem{proposition}[theorem]{Proposition}

\newtheorem{remark}[theorem]{Remark}

\numberwithin{equation}{section}

\newcommand{\BB}{\mathbb{B}}

\newcommand{\op}{\textnormal{op}}

\newcommand{\argmin}{\mathop{\rm argmin}}

\newcommand{\ECal}{\mathcal{E}}

\newcommand{\XCal}{\mathcal{X}}

\newcommand{\br}{\mathbb{R}}

\newcommand{\ba}{\begin{array}}
\newcommand{\ea}{\end{array}}

\newcommand{\FCal}{\mathcal{F}}

\begin{document}


\begin{center}

{\bf{\LARGE{A Control-Theoretic Perspective on Optimal \\ [.2cm]  High-Order Optimization}}}

\vspace*{.2in}
{\large{ \begin{tabular}{c}
Tianyi Lin$^\diamond$ \and Michael I. Jordan$^{\diamond, \dagger}$ \\
\end{tabular}
}}

\vspace*{.2in}

\begin{tabular}{c}
Department of Electrical Engineering and Computer Sciences$^\diamond$ \\
Department of Statistics$^\dagger$ \\
University of California, Berkeley \\
\end{tabular}

\vspace*{.2in}

\today

\vspace*{.2in}

\begin{abstract}
We provide a control-theoretic perspective on optimal tensor algorithms for minimizing a convex function in a finite-dimensional Euclidean space. Given a function $\Phi: \br^d \rightarrow \br$ that is convex and twice continuously differentiable, we study a closed-loop control system that is governed by the operators $\nabla \Phi$ and $\nabla^2 \Phi$ together with a feedback control law $\lambda(\cdot)$ satisfying the algebraic equation $(\lambda(t))^p\|\nabla\Phi(x(t))\|^{p-1} = \theta$ for some $\theta \in (0, 1)$. Our first contribution is to prove the existence and uniqueness of a local solution to this system via the Banach fixed-point theorem. We present a simple yet nontrivial Lyapunov function that allows us to establish the existence and uniqueness of a global solution under certain regularity conditions and analyze the convergence properties of trajectories. The rate of convergence is $O(1/t^{(3p+1)/2})$ in terms of objective function gap and $O(1/t^{3p})$ in terms of squared gradient norm. Our second contribution is to provide two algorithmic frameworks obtained from discretization of our continuous-time system, one of which generalizes the large-step A-HPE framework of~\citet{Monteiro-2013-Accelerated} and the other of which leads to a new optimal $p$-th order tensor algorithm. While our discrete-time analysis can be seen as a simplification and generalization of~\citet{Monteiro-2013-Accelerated}, it is largely motivated by the aforementioned continuous-time analysis, demonstrating the fundamental role that the feedback control plays in optimal acceleration and the clear advantage that the continuous-time perspective brings to algorithmic design. A highlight of our analysis is that we show that all of the $p$-th order optimal tensor algorithms that we discuss minimize the squared gradient norm at a rate of $O(k^{-3p})$, which complements the recent analysis~\citep{Gasnikov-2019-Optimal,Jiang-2019-Optimal,Bubeck-2019-Near}. 
\end{abstract}

\end{center}

\section{Introduction}
The interplay between continuous-time and discrete-time perspectives on dynamical systems has made a major impact on optimization theory. Classical examples include (1) the interpretation of steepest descent, heavy ball and proximal algorithms as the explicit and implicit discretization of gradient-like dissipative systems~\citep{Polyak-1987-Introduction,Antipin-1994-Minimization,Attouch-1996-Dynamical,Alvarez-2000-Minimizing,Attouch-2000-Heavy,Alvarez-2001-Inertial}; and (2) the explicit discretization of Newton-like and Levenberg-Marquardt regularized systems~\citep{Alvarez-1998-Dynamical,Attouch-2001-Second,Alvarez-2002-Second,Attouch-2011-Continuous,Attouch-2012-Second,Mainge-2013-First,Attouch-2013-Global,Abbas-2014-Newton,Attouch-2016-Dynamic,Attouch-2020-Newton,Attouch-2020-Continuous}, which give standard and regularized Newton algorithms. One particularly salient way that these connections have spurred research is via the use of Lyapunov functions to transfer asymptotic behavior and rates of convergence between continuous time and discrete time. 

Recent years have witnessed a flurry of new research focusing on continuous-time perspectives on Nesterov's accelerated gradient algorithm (NAG)~\citep{Nesterov-1983-Method} and related methods~\citep{Guler-1992-New,Beck-2009-Fast,Tseng-2010-Approximation,Nesterov-2013-Gradient}. These perspectives arise from derivations that obtain differential equations as limits of discrete dynamics~\citep{Su-2016-Differential,Krichene-2015-Accelerated,Attouch-2016-Rate,Vassilis-2018-Differential,Shi-2018-Understanding,Muehlebach-2019-Dynamical,Diakonikolas-2019-Approximate,Attouch-2019-Convergence,Sebbouh-2020-Convergence}, including quasi-gradient formulations and Kurdyka-Lojasiewicz theory~\citep{Begout-2015-Damped,Attouch-2020-Fast} (see the references~\citep{Huang-2006-Gradient,Chergui-2008-Convergence,Chill-2010-Gradient,Barta-2012-Every,Barta-2016-Convergence} for geometrical perspective on the topic), inertial gradient systems with constant or asymptotic vanishing damping~\citep{Su-2016-Differential,Attouch-2017-Asymptotic,Attouch-2018-Fast,Attouch-2019-Fast} and their extension to maximally monotone operators~\citep{Bot-2016-Second,Attouch-2018-Convergence,Attouch-2020-Convergence}, Hessian-driven damping~\citep{Alvarez-2002-Second,Attouch-2012-Second,Attouch-2016-Fast,Shi-2018-Understanding,Boct-2020-Tikhonov,Attouch-2020-First,Attouch-2021-Fast}, time scaling~\citep{Attouch-2019-Fast,Attouch-2019-Time,Attouch-2021-Fast,Attouch-2021-ADMM}, dry friction damping~\citep{Adly-2020-Finite,Adly-2021-First}, closed-loop damping~\citep{Attouch-2020-Fast,Attouch-2021-Fast}, control-theoretic design~\citep{Lessard-2016-Analysis,Hu-2017-Dissipativity,Fazlyab-2018-Analysis} and Lagrangian and Hamiltonian frameworks~\citep{Wibisono-2016-Variational,Betancourt-2018-Symplectic,Maddison-2018-Hamiltonian,Donoghue-2019-Hamiltonian,Diakonikolas-2020-Generalized,Francca-2020-Conformal,Muehlebach-2020-Optimization,Francca-2021-Dissipative}. Examples of hitherto unknown results that have arisen from this line of research include the fact that NAG achieves a fast rate of $o(k^{-2})$ in terms of objective function gap~\citep{May-2017-Asymptotic,Attouch-2016-Rate,Attouch-2018-Fast} and $O(k^{-3})$ in terms of squared gradient norm~\citep{Shi-2018-Understanding}. 

The introduction of the Hessian-driven damping into continuous-time dynamics has been a particular milestone in optimization and mechanics. The precursor of this perspective can be found in the variational characterization of the Levenberg-Marquardt method and Newton's method~\citep{Alvarez-1998-Dynamical}, a development that inspired work on continuous-time Newton-like approaches for convex minimization~\citep{Alvarez-1998-Dynamical,Attouch-2001-Second} and monotone inclusions~\citep{Attouch-2011-Continuous,Mainge-2013-First,Attouch-2013-Global,Abbas-2014-Newton,Attouch-2016-Dynamic,Attouch-2020-Newton,Attouch-2020-Continuous}. Building on these works, ~\cite{Alvarez-2002-Second} distinguished Hessian-driven damping from classical continuous Newton formulations and showed its importance in optimization and mechanics. Subsequently, ~\citet{Attouch-2016-Fast} demonstrated the connection between Hessian-driven damping and the forward-backward algorithms in Nesterov acceleration (e.g., FISTA), and combined Hessian-driven damping with asymptotically vanishing damping~\citep{Su-2016-Differential}. The resulting dynamics takes the following form:
\begin{equation}\label{Sys:asymptotic-Hessian}
\ddot{x}(t) + \frac{\alpha}{t}\dot{x}(t) + \beta\nabla^2\Phi(x(t))\dot{x}(t) + \nabla\Phi(x(t)) = 0,
\end{equation}    
where it is worth mentioning that the presence of the Hessian does not entail numerical difficulties since it arises in the form $\nabla^2\Phi(x(t))\dot{x}(t)$, which is the time derivative of the function $t \mapsto \nabla \Phi(x(t))$. Further work in this vein appeared in~\citet{Shi-2018-Understanding}, where Nesterov acceleration was interpreted via multiscale limits that yield high-resolution differential equations:
\begin{equation}\label{Sys:high-resolution}
\ddot{x}(t) + \frac{3}{t}\dot{x}(t) + \sqrt{s}\nabla^2\Phi(x(t))\dot{x}(t) + \left(1+\frac{3\sqrt{s}}{2t}\right)\nabla\Phi(x(t)) = 0.
\end{equation}  
These limits were used in particular to distinguish between Polyak's heavy-ball method and NAG, which are not distinguished by naive limiting arguments that yield the same differential equation for both.

Althought the coefficients are different in Eq.~\eqref{Sys:asymptotic-Hessian} and Eq.~\eqref{Sys:high-resolution}, both contain Hessian-driven damping, which corresponds to a correction term obtained via discretization, and which provides fast convergence to zero of the gradients and reduces the oscillatory aspects. Using this viewpoint, several subtle analyses have been recently provided in work independent of ours~\citep{Attouch-2020-Fast,Attouch-2021-Fast}. In particular, they develop a convergence theory for a general inertial system with asymptotic vanishing damping and Hessian-driven damping. Under certain conditions, the fast convergence is guaranteed in terms of both objective function gap and squared gradient norm. Beyond the aforementioned line of work, however, most of the focus in using continuous-time perspectives to shed light on acceleration has been restricted to the setting of first-order optimization algorithms. As noted in a line of recent work~\citep{Monteiro-2013-Accelerated,Nesterov-2018-Lectures,Arjevani-2019-Oracle,Gasnikov-2019-Optimal,Jiang-2019-Optimal,Bubeck-2019-Near,Song-2021-Unified}, there is a significant gap in our understanding of optimal $p$-th order tensor algorithms with $p \geq 2$, with existing algorithms and analysis being much more involved than NAG.

In this paper, we show that a continuous-time perspective helps to bridge this gap and yields a unified perspective on first-order and higher-order acceleration. We refer to our work as a \emph{control-theoretic perspective}, as it involves the study of a closed-loop control system that can be viewed as a differential equation that is governed by a feedback control law, $\lambda(\cdot)$, satisfying the algebraic equation $(\lambda(t))^p\|\nabla\Phi(x(t))\|^{p-1} = \theta$ for some $\theta \in (0, 1)$. Our approach is similar to that of~\citet{Attouch-2013-Global,Attouch-2016-Dynamic}, for the case without inertia, and it provides a first step into a theory of the autonomous inertial systems that link closed-loop control and optimal high-order tensor algorithms. Mathematically, our system can be written as follows:
\begin{equation}\label{Sys:general}
\ddot{x}(t) + \alpha(t)\dot{x}(t) + \beta(t)\nabla^2\Phi(x(t))\dot{x}(t) + b(t)\nabla\Phi(x(t)) = 0,
\end{equation}
where $(\alpha, \beta, b)$ explicitly depends on the variables $(x, \lambda, a)$, the parameters $c > 0$, $\theta \in (0, 1)$ and the order $p \in \{1, 2, \ldots\}$:  
\begin{equation}\label{Sys:choice-feedback}
\begin{array}{ll}
& \alpha(t) = \frac{2\dot{a}(t)}{a(t)} - \frac{\ddot{a}(t)}{\dot{a}(t)}, \quad \beta(t) = \frac{(\dot{a}(t))^2}{a(t)}, \quad b(t) = \frac{\dot{a}(t)(\dot{a}(t) + \ddot{a}(t))}{a(t)}, \\
& a(t) = \frac{1}{4}(\int_0^t \sqrt{\lambda(s)} ds + c)^2, \quad (\lambda(t))^p\|\nabla\Phi(x(t))\|^{p-1} = \theta.   
\end{array} 
\end{equation}
The initial condition is $x(0) = x_0 \in \{x \in \br^d \mid \|\nabla\Phi(x)\| \neq 0\}$ and $\dot{x}(0) \in \br^d$. Note that this condition is not restrictive since $\|\nabla\Phi(x_0)\| = 0$ implies that the optimization problem has been already solved. A key ingredient in our system is the algebraic equation $(\lambda(t))^p\|\nabla\Phi(x(t))\|^{p-1} = \theta$, which links the feedback control law $\lambda(\cdot)$ and the gradient norm $\|\nabla\Phi(x(\cdot))\|$, and which generalizes an equation appearing in~\citet{Attouch-2016-Dynamic} for modeling the proximal Newton algorithm. We recall that Eq.~\eqref{Sys:general} has also been studied in~\citet{Attouch-2020-Fast,Attouch-2021-Fast}, who provide a general convergence result when $(\alpha, \beta, b)$ satisfies certain conditions. However, when $p \geq 2$, the specific choice of $(\alpha, \beta, b)$ in Eq.~\eqref{Sys:choice-feedback} does not have an analytic form and it thus seems difficult to verify whether $(\alpha, \beta, b)$ in our control system satisfies that condition (see~\citet[Theorem~2.1]{Attouch-2021-Fast})). This topic is beyond the scope of this paper and we leave its investigation to future work.  

\paragraph{Our contribution.} Throughout the paper, unless otherwise indicated, we assume that 
\begin{quote}
\centering
\textit{$\Phi: \br^d \rightarrow \br$ is convex and twice continuously differentiable and the set of global minimizers of $\Phi$} is nonempty.
\end{quote}
As we shall see, our main results on the existence and uniqueness of solutions and convergence properties of trajectories are valid under this general assumption. We also believe that this general setting paves the way for extensions to nonsmooth convex functions or maximal monotone operators (replacing the gradient by the subdifferential or the operator)~\citep{Alvarez-2002-Second,Attouch-2012-Second,Attouch-2016-Fast}. This is evidenced by the equivalent first-order reformulations of our closed-loop control system in time and space (without the occurrence of the Hessian). However, we do not pursue these extensions in the current paper. 

The main contributions of our work are the following: 
\begin{enumerate}
\item We study the closed-loop control system of Eq.~\eqref{Sys:general} and Eq.~\eqref{Sys:choice-feedback} and prove the existence and uniqueness of a local solution. We show that when $p = 1$ and $c = 0$, our feedback law reduces to $\lambda(t) = \theta$ and our overall system reduces to the high-resolution differential equation studied in~\citet{Shi-2018-Understanding}, showing explicitly that our system extends the high-resolution framework from first-order optimization to high-order optimization.
 
\item We construct a simple yet nontrivial Lyapunov function that allows us to establish the existence and uniqueness of a global solution under regularity conditions (see Theorem~\ref{Theorem:Global-Existence-Uniquess}). We also use the Lyapunov function to analyze the convergence rates of the solution trajectories; in particular, we show that the convergence rate is $O(t^{-(3p+1)/2})$ in terms of objective function gap and $O(t^{-3p})$ in terms of squared gradient norm. 

\item We provide two algorithmic frameworks based on the implicit discretization of our closed-looped control system, one of which generalizes the \textsf{large-step A-HPE} in~\citet{Monteiro-2013-Accelerated}. Our iteration complexity analysis is largely motivated by the aforementioned continuous-time analysis, simplifying the analysis in~\citet{Monteiro-2013-Accelerated} for the case of $p=2$ and generalizing  it to $p > 2$ in a systematic manner (see Theorem~\ref{Theorem:CAFI-Main} and~\ref{Theorem:CAFII-Main} for the details).

\item We combine the algorithmic frameworks with an approximate tensor subroutine, yielding a suite of optimal $p$-th order tensor algorithms for minimizing a convex smooth function $\Phi$ which has Lipschitz $p$-th order derivatives. The resulting algorithms include not only the algorithms studied in the previous works~\citep{Gasnikov-2019-Optimal,Jiang-2019-Optimal,Bubeck-2019-Near} but also yield a new optimal $p$-th order tensor algorithm. A highlight of our analysis is to show that all these $p$-th order optimal algorithms minimize the squared gradient norm at a rate of $O(k^{-3p})$, complementing the recent analysis in the aforementioned works.   
\end{enumerate}

\paragraph{Further related work.} In addition to the aforementioned works, we provide a few additional remarks regarding related work on accelerated first-order and high-order algorithms for convex optimization. 

A significant body of recent work in convex optimization focuses on understanding the underlying principle behind Nesterov's accelerated first-order algorithm (NAG)~\citep{Nesterov-1983-Method,Nesterov-2018-Lectures}, with a particular focus on the interpretation of Nesterov acceleration as a temporal discretization of a continuous-time dynamical system~\citep{Krichene-2015-Accelerated,Su-2016-Differential,Attouch-2016-Rate,May-2017-Asymptotic,Attouch-2019-Rate,Attouch-2018-Fast,Vassilis-2018-Differential,Shi-2018-Understanding,Attouch-2018-Fast,Diakonikolas-2019-Approximate,Muehlebach-2019-Dynamical,Attouch-2019-Convergence,Attouch-2019-Fast,Sebbouh-2020-Convergence,Attouch-2020-Convergence,Attouch-2020-First,Attouch-2020-Fast,Attouch-2021-Fast,Adly-2021-First}. A line of new first-order algorithms have been obtained from the continuous-time dynamics by various advanced numerical integration strategies~\citep{Scieur-2017-Integration,Betancourt-2018-Symplectic,Zhang-2018-Direct,Maddison-2018-Hamiltonian,Shi-2019-Acceleration,Wilson-2019-Accelerating}. In particular,~\citet{Scieur-2017-Integration} showed that a basic gradient flow system and multi-step integration scheme yields a class of accelerated first-order optimization algorithms.~\citet{Zhang-2018-Direct} applied Runge-Kutta integration to an inertial gradient system without Hessian-driven damping~\citep{Wibisono-2016-Variational} and showed that the resulting algorithm is faster than NAG when the objective function is sufficiently smooth and when the order of the integrator is sufficiently large.~\citet{Maddison-2018-Hamiltonian} and~\citet{Francca-2020-Conformal} both considered conformal Hamiltonian systems and showed that the resulting discrete-time algorithm achieves fast convergence under certain smoothness conditions. Very recently,~\citet{Shi-2019-Acceleration} have rigorously justified the use of symplectic Euler integrators compared to explicit and implicit Euler integration, which was further studied by~\citet{Muehlebach-2020-Optimization} and~\citet{Francca-2021-Dissipative}. Unfortunately, none of these approaches are suitable for interpreting optimal acceleration in high-order tensor algorithms. 

Research on acceleration in the second-order setting dates back to Nesterov's accelerated cubic regularized Newton algorithm (ACRN)~\citep{Nesterov-2008-Accelerating} and Monteiro and Svaiter's accelerated Newton proximal extragradient (A-NPE)~\citep{Monteiro-2013-Accelerated}. The ACRN algorithm was extended to a $p$-th order tensor algorithm with the improved convergence rate of $O(k^{-(p+1)})$~\citep{Baes-2009-Estimate} and an adaptive $p$-th order tensor algorithm with essentially the same rate~\citep{Jiang-2020-Unified}. This novel extension was also revisited by~\citet{Nesterov-2019-Implementable} with a discussion on the efficient implementation of a third-order tensor algorithm. Meanwhile, within the alternative A-NPE framework, a $p$-th order tensor algorithm was studied in a line of works~\citep{Gasnikov-2019-Optimal,Jiang-2019-Optimal,Bubeck-2019-Near} and was shown to achieve a convergence rate of $O(k^{-(3p+1)/2})$, matching the lower bound~\citep{Arjevani-2019-Oracle}. Subsequently, a high-order coordinate descent algorithm was studied in~\citet{Amaral-2020-Complexity}, and very recently, the high-order A-NPE framework has been specialized to the strongly convex setting~\citep{Alves-2021-Variants}, generalizing the discrete-time algorithms in this paper with an improved convergence rate. Beyond the setting of Lipschitz continuous derivatives, high-order algorithms and their accelerated variants have been adapted for more general setting with H\"{o}lder continuous derivatives~\citep{Grapiglia-2017-Regularized,Grapiglia-2019-Accelerated,Doikov-2019-Local,Grapiglia-2020-Tensor,Grapiglia-2020-Stationary} and an optimal algorithm has been proposed in~\citet{Song-2021-Unified}. Other settings include structured convex non-smooth minimization~\citep{Bullins-2020-Highly}, convex-concave minimax optimization and monotone variational inequalities~\citep{Bullins-2020-Higher,Ostroukhov-2020-Tensor}, and structured smooth convex minimization~\citep{Nesterov-2020-Superfast,Nesterov-2020-Inexact,Kamzolov-2020-Near,Kamzolov-2020-Hyperfast}. In the nonconvex setting, high-order algorithms have been also proposed and analyzed~\citep{Birgin-2016-Evaluation,Birgin-2017-Worst,Martinez-2017-High,Cartis-2018-Second,Cartis-2019-Universal}.

Unfortunately, the derivations of these algorithms do not flow from a single underlying principle but tend to involve case-specific algebra. As in the case of first-order algorithms, one would hope that a continuous-time perspective would offer unification, but the only work that we are aware of in this regard is~\citet{Song-2021-Unified}, and the connection to dynamical systems in that work is unclear. In particular, some aspects of the UAF algorithm (see~\citet[Algorithm~5.1]{Song-2021-Unified}), including the conditions in Eq. (5.31) and Eq. (5.32), do not have a continuous-time interpretation but rely on case-specific algebra. Moreover, their continuous-time framework reduces to an inertial system without Hessian-driven damping in the first-order setting, which has been proven to be an inaccurate surrogate as mentioned earlier.

We have been also aware of other type of discrete-time algorithms~\citep{Zhang-2018-Direct,Maddison-2018-Hamiltonian,Wilson-2019-Accelerating} which were derived from continuous-time perspective with theoretical guarantee under certain condition. In particular,~\citet{Wilson-2019-Accelerating} derived a family of first-order algorithms by appeal to the explicit time discretization of the accelerated rescaled gradient dynamics. Their new algorithms are guaranteed to (surprisingly) achieve the same convergence rate as the existing optimal tensor algorithms~\citep{Gasnikov-2019-Optimal,Jiang-2019-Optimal,Bubeck-2019-Near}. However, the strong smoothness assumption is necessary and might rule out many interesting application problems. In contrast, all the optimization algorithms developed in this paper are applicable for \textit{general} convex and smooth problems with the optimal rate of convergence.

\paragraph{Organization.} The remainder of the paper is organized as follows. In Section~\ref{sec:control-system}, we study the closed-loop control system in Eq.~\eqref{Sys:general} and Eq.~\eqref{Sys:choice-feedback} and prove the existence and uniqueness of a local solution using the Banach fixed-point theorem. In Section~\ref{sec:Lyapunov}, we show that our system permits a simple yet nontrivial Lyapunov function which allows us to establish the existence and uniqueness of a global solution and derive convergence rates of solution trajectories. In Section~\ref{sec:optimal-algorithm}, we provide two conceptual algorithmic frameworks based on the implicit discretization of our closed-loop control system as well as specific optimal $p$-th order tensor algorithms. Our iteration complexity analysis is largely motivated by the continuous-time analysis of our system, demonstrating that these algorithms achieve fast gradient minimization. In Section~\ref{sec:conclusions}, we conclude our work with a brief discussion on future research directions.

\paragraph{Notation.} We use bold lower-case letters such as $x$ to denote vectors, and upper-case letters such as $X$ to denote tensors. For a vector $x \in \br^d$, we let $\|x\|$ denote its $\ell_2$ Euclidean norm and let $\BB_\delta(x) = \{x' \in \br^d \mid \|x'-x\| \leq \delta\}$ denote its $\delta$-neighborhood. For a tensor $X \in \br^{d_1 \times \cdots \times d_p}$, we define 
\begin{equation*}
X[z^1, \cdots, z^p] = \sum_{1 \leq i_j \leq d_j, 1 \leq j \leq p} \left[X_{i_1, \cdots, i_p}\right]z_{i_1}^1 \cdots z_{i_p}^p, 
\end{equation*}
and denote by $\|X\|_\op = \max_{\|z^i\|=1, 1 \leq j \leq p} X[z^1, \cdots, z^p]$ its operator norm. 

Fix $p \geq 1$, we define $\FCal_\ell^p(\br^d)$ as the class of convex functions on $\br^d$ with $\ell$-Lipschitz $p$-th order derivatives; that is, $f \in \FCal_\ell^p(\br^d)$ if and only if $f$ is convex and $\|\nabla^{(p)} f(x') - \nabla^{(p)} f(x)\|_\op \leq \ell\|x' - x\|$ for all $x, x' \in \br^d$ in which $\nabla^{(p)} f(x)$ is the $p$-th order derivative tensor of $f$ at $x \in \br^d$. More specifically, for $\{z^1, z^2, \ldots, z^p\} \subseteq \br^d$, we have
\begin{equation*}
\nabla^{(p)} f(x)[z^1, \ldots, z^p] = \sum_{1 \leq i_1, \cdots, i_p \leq d} \left[\frac{\partial^p f}{\partial x_{i_1} \cdots \partial x_{i_p}}(x)\right]z_{i_1}^1 \cdots z_{i_p}^p. 
\end{equation*}
Given a tolerance $\epsilon \in (0, 1)$, the notation $a = O(b(\epsilon))$ stands for an upper bound, $a \leq C b(\epsilon)$, in which $C > 0$ is independent of $\epsilon$. 

\section{The Closed-Loop Control System}\label{sec:control-system}
In this section, we study the closed-loop control system in Eq.~\eqref{Sys:general} and Eq.~\eqref{Sys:choice-feedback}. We start by rewriting our system as a first-order system in time and space (without the occurrence of the Hessian) which is important to our subsequent analysis and implicit time discretization. Then, we analyze the algebraic equation $(\lambda(t))^p\|\nabla\Phi(x(t))\|^{p-1} = \theta$ for $\theta \in (0, 1)$ and prove the existence and uniqueness of a local solution by appeal to the Banach fixed-point theorem. We conclude by discussing other systems in the literature that exemplify our general framework.

\subsection{First-order system in time and space}\label{subsec:FO}
We rewrite the closed-loop control system in Eq.~\eqref{Sys:general} and Eq.~\eqref{Sys:choice-feedback} as follows:
\begin{equation*}
\ddot{x}(t) + \alpha(t)\dot{x}(t) + \beta(t)\nabla^2\Phi(x(t))\dot{x}(t) + b(t)\nabla\Phi(x(t)) = 0,
\end{equation*}
where $(\alpha, \beta, b)$ explicitly depend on the variables $(x, \lambda, a)$, the parameters $c > 0$, $\theta \in (0, 1)$ and the order $p \in \{1, 2, \ldots\}$:  
\begin{equation*}
\begin{array}{ll}
& \alpha(t) = \frac{2\dot{a}(t)}{a(t)} - \frac{\ddot{a}(t)}{\dot{a}(t)}, \quad \beta(t) = \frac{(\dot{a}(t))^2}{a(t)}, \quad b(t) = \frac{\dot{a}(t)(\dot{a}(t) + \ddot{a}(t))}{a(t)}, \\
& a(t) = \frac{1}{4}(\int_0^t \sqrt{\lambda(s)} ds + c)^2, \quad (\lambda(t))^p\|\nabla\Phi(x(t))\|^{p-1} = \theta.   
\end{array} 
\end{equation*}
By multiplying both sides of the first equation by $\frac{a(t)}{\dot{a}(t)}$ and using the definition of $\alpha(t)$, $\beta(t)$ and $b(t)$, we have 
\begin{equation*}
\frac{a(t)}{\dot{a}(t)}\ddot{x}(t) + \left(2 - \frac{a(t)\ddot{a}(t)}{(\dot{a}(t))^2}\right)\dot{x}(t) + \dot{a}(t)\nabla^2\Phi(x(t))\dot{x}(t) + (\dot{a}(t) + \ddot{a}(t))\nabla\Phi(x(t)) = 0. 
\end{equation*}
Defining $z_1(t) = \frac{a(t)}{\dot{a}(t)}\dot{x}(t)$ and $z_2(t) = \dot{a}(t)\nabla\Phi(x(t))$, we have
\begin{equation*}
\dot{z}_1(t) = \frac{a(t)}{\dot{a}(t)}\ddot{x}(t) + \left(1 - \frac{a(t)\ddot{a}(t)}{(\dot{a}(t))^2}\right)\dot{x}(t), \quad \dot{z}_2(t) = \dot{a}(t)\nabla^2\Phi(x(t))\dot{x}(t) + \ddot{a}(t)\nabla\Phi(x(t)). 
\end{equation*}
Putting these pieces together yields 
\begin{equation*}
\dot{z}_1(t) + \dot{x}(t) + \dot{z}_2(t) = -\dot{a}(t)\nabla\Phi(x(t)). 
\end{equation*}
Integrating this equation over the interval $[0, t]$, we have
\begin{equation}\label{Prelim:FO-first}
z_1(t) + x(t) + z_2(t) = z_1(0) + x(0) + z_2(0) - \int_0^t \dot{a}(s)\nabla\Phi(x(s)) ds. 
\end{equation}
Since $x(0) = x_0 \in \{x \in \br^d \mid \|\nabla\Phi(x)\| \neq 0\}$, it is easy to verify that $\lambda(0)$ is well defined and determined by the algebraic equation $\lambda(0) = \theta^{\frac{1}{p}}\|\nabla\Phi(x_0)\|^{-\frac{p-1}{p}}$. Using the definition of $a(t)$, we have $a(0) = \frac{c^2}{4}$ and $\dot{a}(0) = \frac{c\theta^{\frac{1}{2p}}\|\nabla\Phi(x_0)\|^{-\frac{p-1}{2p}}}{2}$. Putting these pieces together with the definition of $z_1(t)$ and $z_2(t)$, we have
\begin{eqnarray*}
z_1(0) + x(0) + z_2(0) & = & \frac{a(0)}{\dot{a}(0)}\dot{x}(0) + x(0) + \dot{a}(0)\nabla\Phi(x(0)) \\
& & \hspace*{-8em} = \ x(0) + \frac{c\theta^{-\frac{1}{2p}}\dot{x}(0)\|\nabla\Phi(x(0))\|^{\frac{p-1}{2p}} + c\theta^{\frac{1}{2p}}\|\nabla\Phi(x(0))\|^{-\frac{p-1}{2p}}\nabla\Phi(x(0))}{2}. 
\end{eqnarray*}
This implies that $z_1(0) + x(0) + z_2(0)$ is completely determined by the initial condition and parameters $c>0$ and $\theta \in (0, 1)$. For simplicity, we define $v_0 := z_1(0) + x(0) + z_2(0)$ and rewrite Eq.~\eqref{Prelim:FO-first} in the following form: 
\begin{equation}\label{Prelim:FO-second}
\frac{a(t)}{\dot{a}(t)}\dot{x}(t) + x(t) + \dot{a}(t)\nabla\Phi(x(t)) = v_0 - \int_0^t \dot{a}(s)\nabla\Phi(x(s)) ds.  
\end{equation}
By introducing a new variable $v(t) = v_0 - \int_0^t \dot{a}(s)\nabla\Phi(x(s)) ds$, we rewrite Eq.~\eqref{Prelim:FO-second} in the following equivalent form:
\begin{equation*}
\dot{v}(t) + \dot{a}(t)\nabla \Phi(x(t)) = 0, \quad \dot{x}(t) + \frac{\dot{a}(t)}{a(t)}(x(t) - v(t)) + \frac{(\dot{a}(t))^2}{a(t)}\nabla\Phi(x(t)) = 0.  
\end{equation*}
Summarizing, the closed-loop control system in Eq.~\eqref{Sys:general} and Eq.~\eqref{Sys:choice-feedback} can be written as a first-order system in time and space as follows:
\begin{equation}\label{Sys:FO}
\left\{\begin{array}{ll}
& \dot{v}(t) + \dot{a}(t)\nabla \Phi(x(t)) = 0 \\ 
& \dot{x}(t) + \frac{\dot{a}(t)}{a(t)}(x(t) - v(t)) + \frac{(\dot{a}(t))^2}{a(t)}\nabla\Phi(x(t)) = 0 \\
& a(t) = \frac{1}{4}(\int_0^t \sqrt{\lambda(s)} ds + c)^2 \\
& (\lambda(t))^p\|\nabla\Phi(x(t))\|^{p-1} = \theta \\ 
& (x(0), v(0)) = (x_0, v_0).
\end{array} \right.
\end{equation}
We also provide another first-order system in time and space with different variable $(x, v, \lambda, \gamma)$. We study this system because its implicit time discretization leads to a new algorithmic framework which does not appear in the literature. This first-order system is summarized as follows:
\begin{equation}\label{Sys:FO-inverse}
\left\{\begin{array}{ll}
& \dot{v}(t) - \frac{\dot{\gamma}(t)}{\gamma^2(t)}\nabla \Phi(x(t)) = 0 \\
& \dot{x}(t) - \frac{\dot{\gamma}(t)}{\gamma(t)}(x(t) - v(t)) +  \frac{(\dot{\gamma}(t))^2}{(\gamma(t))^3}\nabla\Phi(x(t)) = 0 \\ 
& \gamma(t) = 4(\int_0^t \sqrt{\lambda(s)} ds + c)^{-2} \\
& (\lambda(t))^p\|\nabla\Phi(x(t))\|^{p-1} = \theta \\
& (x(0), v(0)) = (x_0, v_0).
\end{array} \right.
\end{equation}
\begin{remark}
The first-order systems in Eq.~\eqref{Sys:FO} and Eq.~\eqref{Sys:FO-inverse} are equivalent. It suffices to show that 
\begin{equation*}
\dot{a}(t) = - \frac{\dot{\gamma}(t)}{\gamma^2(t)}, \quad \frac{\dot{a}(t)}{a(t)} = - \frac{\dot{\gamma}(t)}{\gamma(t)}, \quad \frac{(\dot{a}(t))^2}{a(t)} = \frac{(\dot{\gamma}(t))^2}{(\gamma(t))^3}. 
\end{equation*}
By the definition of $a(t)$ and $\gamma(t)$, we have $a(t) = \frac{1}{\gamma(t)}$ which implies that $\dot{a}(t) = - \frac{\dot{\gamma}(t)}{\gamma^2(t)}$. 
\end{remark}
\begin{remark}
The first-order systems in Eq.~\eqref{Sys:FO} and Eq.~\eqref{Sys:FO-inverse} pave the way for extensions to nonsmooth convex functions or maximal monotone operators (replacing the gradient by the subdifferential or the operator), as done in~\citet{Alvarez-2002-Second} and~\citet{Attouch-2012-Second,Attouch-2016-Fast}. In this setting, either the open-loop case or the closed-loop case without inertia has been studied in the literature~\citep{Attouch-2011-Continuous,Mainge-2013-First,Attouch-2013-Global,Abbas-2014-Newton,Attouch-2016-Dynamic,Bot-2016-Second,Attouch-2018-Convergence,Attouch-2020-Convergence,Attouch-2020-Newton}, but there is significantly less work on the case of a closed-loop control system with inertia. For recent progress in this direction, see~\citet{Attouch-2020-Fast} and references therein.  
\end{remark}

\subsection{Algebraic equation}\label{subsec:AE}
We study the algebraic equation,
\begin{equation}\label{Eq:AE-main}
(\lambda(t))^p\|\nabla\Phi(x(t))\|^{p-1} = \theta \in (0, 1),  
\end{equation}
which links the feedback control $\lambda(\cdot)$ and the solution trajectory $x(\cdot)$ in the closed-loop control system. To streamline the presentation, we define a function $\varphi: [0, +\infty) \times \br^d \mapsto [0, +\infty)$ such that 
\begin{equation*}
\varphi(\lambda, x) = \lambda\|\nabla\Phi(x)\|^{\frac{p-1}{p}}, \quad \varphi(0, x) = 0. 
\end{equation*}
By definition, Eq.~\eqref{Eq:AE-main} is equivalent to $\varphi(\lambda(t), x(t)) = \theta^{1/p}$. Our first proposition presents a property of the mapping $\varphi(\cdot, x)$, for a fixed $x \in \br^d$ satisfying $\nabla \Phi(x) \neq 0$.  We have:
\begin{proposition}\label{Prop:AE-mapping-first}
Fixing $x \in \br^d$ with $\nabla \Phi(x) \neq 0$, the mapping $\varphi(\cdot, x)$ satisfies
\begin{enumerate}
\item $\varphi(\cdot, x)$ is linear, strictly increasing and $\varphi(0, x) = 0$. 
\item $\varphi(\lambda, x) \rightarrow +\infty$ as $\lambda \rightarrow +\infty$. 
\end{enumerate}
\end{proposition}
\begin{proof}
By the definition of $\varphi$, the mapping $\varphi(\cdot, x)$ is linear and $\varphi(0, x) = 0$. Since $\nabla \Phi(x) \neq 0$, we have $\|\nabla \Phi(x)\| > 0$ and $\varphi(\cdot, x)$ is thus strictly increasing. Since $\varphi(\cdot, x)$ is linear and strictly increasing, $\varphi(\lambda, x) \rightarrow +\infty$ as $\lambda \rightarrow +\infty$. 
\end{proof}

In view of Proposition~\ref{Prop:AE-mapping-first}, for any fixed point $x$ with $\nabla \Phi(x) \neq 0$, there exists a unique $\lambda > 0$ such that $\varphi(\lambda, x) = \theta^{1/p}$ for some $\theta \in (0, 1)$. We accordingly define $\Omega \subseteq \br^d$ and the mapping $\Lambda_\theta: \Omega \mapsto (0, \infty)$ as follows:
\begin{equation}\label{Def:AE-mapping}
\Omega = \{x \in \br^d \mid \|\nabla\Phi(x)\| \neq 0\}, \quad \Lambda_\theta(x) = \theta^{\frac{1}{p}}\|\nabla\Phi(x)\|^{-\frac{p-1}{p}}. 
\end{equation}
We now provide several basic results concerning $\Omega$ and $\Lambda_\theta(\cdot)$ which are crucial to the proof of existence and uniqueness presented in the next subsection.
\begin{proposition}\label{Prop:AE-mapping-second}
The set $\Omega$ is open. 
\end{proposition}
\begin{proof}
Given $x \in \Omega$, it suffices to show that $\BB_\delta(x) \subseteq \Omega$ for some $\delta>0$. Since $\Phi$ is twice continuously differentiable, $\nabla\Phi$ is locally Lipschitz; that is, there exists $\tilde{\delta} > 0$ and $L > 0$ such that 
\begin{equation*}
\|\nabla\Phi(z) - \nabla\Phi(x)\| \leq L\|z - x\|, \quad \forall z \in \BB_{\delta_1}(x). 
\end{equation*} 
Combining this inequality with the triangle inequality, we have
\begin{equation*}
\|\nabla\Phi(z)\| = \|\nabla\Phi(x)\| - \|\nabla\Phi(z) - \nabla\Phi(x)\| \geq \|\nabla\Phi(x)\| - L\|z - x\|. 
\end{equation*}
Let $\delta = \min\{\tilde{\delta}, \frac{\|\nabla\Phi(x)\|}{2L}\}$. Then, for any $z \in \BB_\delta(x)$, we have
\begin{equation*}
\|\nabla\Phi(z)\| \geq \frac{\|\nabla\Phi(x)\|}{2} > 0 \ \Longrightarrow \    z \in \Omega. 
\end{equation*} 
This completes the proof. 
\end{proof}
\begin{proposition}\label{Prop:AE-mapping-Key}
Fixing $\theta \in (0, 1)$, the mappings $\Lambda_\theta(\cdot)$ and $\sqrt{\Lambda_\theta(\cdot)}$ are continuous and locally Lipschitz over $\Omega$. 
\end{proposition}
\begin{proof}
By the definition of $\Lambda_\theta(\cdot)$, it suffices to show that $\Lambda_\theta(\cdot)$ is continuous and locally Lipschitz over $\Omega$ since the same argument works for $\sqrt{\Lambda_\theta(\cdot)}$. 

First, we prove the continuity of $\Lambda_\theta(\cdot)$ over $\Omega$. Since $\|\nabla\Phi(x)\| > 0$ for any $x \in \Omega$, the function $\|\nabla\Phi(\cdot)\|^{-\frac{p-1}{p}}$ is continuous over $\Omega$. By the definition of $\Lambda_\theta(\cdot)$, we achieve the desired result.  

Second, we prove that $\Lambda_\theta(\cdot)$ is locally Lipschitz over $\Omega$. Since $\Phi$ is twice continuously differentiable, $\nabla\Phi$ is locally Lipschitz. For $p = 1$, $\Lambda_\theta(\cdot)$ is a constant everywhere and thus locally Lipschitz over $\Omega$. For $p \geq 2$, the function $x^{-\frac{p-1}{p}}$ is locally Lipschitz at any point $x > 0$. Also, by Proposition~\ref{Prop:AE-mapping-second}, $\Omega$ is an open set. Putting these pieces together yields that $\|\nabla\Phi(\cdot)\|^{-\frac{p-1}{p}}$ is locally Lipschitz over $\Omega$; that is, there exist $\delta > 0$ and $L > 0$ such that
\begin{equation*}
|\|\nabla \Phi(x')\|^{-\frac{p-1}{p}} - \|\nabla \Phi(x'')\|^{-\frac{p-1}{p}}| \leq L\|x' - x''\|, \quad \forall x', x'' \in \BB_\delta(x),
\end{equation*}
which implies that  
\begin{equation*}
|\Lambda_\theta(x') - \Lambda_\theta(x'')| \leq \theta^{\frac{1}{p}}L\|x' - x''\|, \quad \forall x', x'' \in \BB_\delta(x). 
\end{equation*}
This completes the proof. 
\end{proof}

\subsection{Existence and uniqueness of a local solution}\label{subsec:LEU}
We prove the existence and uniqueness of a local solution of the closed-loop control system in Eq.~\eqref{Sys:general} and Eq.~\eqref{Sys:choice-feedback} by appeal to the Banach fixed-point theorem. Using the results in Section~\ref{subsec:FO} (see Eq.~\eqref{Prelim:FO-second}), our system can be equivalently written as follows:
\begin{equation*}
\left\{\begin{array}{ll}
& \dot{x}(t) + \frac{\dot{a}(t)}{a(t)}(x(t) + \int_0^t \dot{a}(s)\nabla\Phi(x(s)) ds - v_0) + \frac{(\dot{a}(t))^2}{a(t)}\nabla\Phi(x(t)) = 0 \\
& a(t) = \frac{1}{4}(\int_0^t \sqrt{\lambda(s)} ds + c)^2 \\
& (\lambda(t))^p\|\nabla\Phi(x(t))\|^{p-1} = \theta \\ 
& x(0) = x_0. 
\end{array}\right.  
\end{equation*}
Using the mapping $\Lambda_\theta: \Omega \mapsto (0, \infty)$ (see Eq.~\eqref{Def:AE-mapping}), this system can be further formulated as an autonomous system. Indeed, we have
\begin{equation*}
\lambda(t) = \Lambda_\theta(x(t)) \Longleftrightarrow [\lambda(t)]^p\|\nabla\Phi(x(t))\|^{p-1} = \theta,   
\end{equation*}
which implies that 
\begin{equation*}
a(t) = \frac{1}{4}\left(\int_0^t \sqrt{\Lambda_\theta(x(s))} ds + c\right)^2, \quad \dot{a}(t) = \frac{1}{2}\sqrt{\Lambda_\theta(x(t))}\left(\int_0^t \sqrt{\Lambda_\theta(x(s))} \ ds + c\right).  
\end{equation*}
Putting these pieces together, we arrive at an autonomous system in the following compact form: 
\begin{equation}\label{Sys:autonomous}
\dot{x}(t) = F(t, x(t)), \quad x(0) = x_0 \in \Omega,
\end{equation}
where the vector field $F: [0, +\infty) \times \Omega \mapsto \br^d$ is given by
\begin{equation}\label{Sys:vector-field}
\begin{array}{lcl}
F(t, x(t)) & = & - \frac{\sqrt{\Lambda_\theta(x(t))}(2x(t) + \int_0^t \sqrt{\Lambda_\theta(x(s))}(\int_0^s \sqrt{\Lambda_\theta(x(w))} dw + c)\nabla \Phi(x(s))ds - v_0)}{\int_0^t \sqrt{\Lambda_\theta(x(s))} \ ds + c} \\ 
& & - \Lambda_\theta(x(t))\nabla\Phi(x(t)). 
\end{array}
\end{equation}
A common method for proving the existence and uniqueness of a local solution is via appeal to the Cauchy-Lipschitz theorem~\citep[Theorem I.3.1]{Coddington-1955-Theory}.  This theorem, however, requires that $F(t, x)$ be continuous in $t$ and Lipschitz in $x$, and this is not immediate in our case due to the appearance of $\int_0^t \sqrt{\Lambda_\theta(x(s))} ds$. We instead recall that the proof of the Cauchy-Lipschitz theorem is generally based on the Banach fixed-point theorem~\citep{Granas-2013-Fixed}, and we avail ourselves directly of the latter theorem.  In particular, we construct Picard iterates $\psi_k$ whose limit is a fixed point of a contraction $T$. We have the following theorem. 
\begin{theorem}\label{Theorem:Local-Existence-Uniquess} 
There exists $t_0 > 0$ such that the autonomous system in Eq.~\eqref{Sys:autonomous} and Eq.~\eqref{Sys:vector-field} has a unique solution $x: [0, t_0] \mapsto \br^d$.
\end{theorem}
\begin{proof}
By Proposition~\ref{Prop:AE-mapping-second} and the initial condition $x_0 \in \Omega$, there exists $\delta > 0$ such that $\BB_{\delta}(x_0) \subseteq \Omega$. Note that $\Phi$ is twice continuously differentiable. By the definition of $\Lambda_\theta$, we obtain that $\Lambda_\theta(z)$ and $\nabla\Phi(z)$ are both bounded for any $z \in \BB_{\delta}(x_0)$. Putting these pieces together shows that there exists $M > 0$ such that, for any continuous function $x: [0, 1] \mapsto \BB_{\delta}(x_0)$, we have
\begin{equation}\label{inequality:LEU-first}
\|F(t, x(t))\| \leq M, \quad \forall t \in [0, 1].  
\end{equation}
The set of such functions is not empty since a constant function $x = x_0$ is one element. Letting $t_1 = \min\{1, \frac{\delta}{M}\}$, we define $\XCal$ as the space of all continuous functions $x$ on $[0, t_0]$ for some $t_0 < t_1$ whose graph is contained entirely inside the rectangle $[0, t_0] \times \BB_{\delta}(x_0)$. For any $x \in \XCal$, we define 
\begin{equation*}
z(t) = Tx = x_0 + \int_0^t F(s, x(s)) ds. 
\end{equation*}
Note that $z(\cdot)$ is well defined and continuous on $[0, t_0]$. Indeed, $x \in \XCal$ implies that $x(t) \in \BB_{\delta}(x_0) \subseteq \Omega$ for $\forall t \in [0, t_0]$. Thus, the integral of $F(s, x(s))$ is well defined and continuous. Second, the graph of $z(t)$ lies entirely inside the rectangle $[0, t_0] \times \BB_{\delta}(x_0)$. Indeed, since $t \leq t_0 < t_1 = \min\{1, \frac{\delta}{M}\}$, we have
\begin{equation*}
\|z(t) - x_0\| = \left\|\int_0^t F(s, x(s)) ds\right\| \overset{\textnormal{Eq.~\eqref{inequality:LEU-first}}}{\leq} Mt \leq Mt_0 \leq Mt_1 \leq \delta. 
\end{equation*}
Putting these pieces together yields that $T$ maps $\XCal$ to itself. By the fundamental theorem of calculus, we have $\dot{z}(t) = F(t, x(t))$. By a standard argument from ordinary differential equation theory, $\dot{x}(t) = F(t, x(t))$ and $x(0)=x_0$ if and only if $x$ is a fixed point of $T$. Thus, it suffices to show the existence and uniqueness of a fixed point of $T$. 

We consider the Picard iterates $\{\psi_k\}_{k \geq 0}$ with $\psi_0(t) = x_0$ for $\forall t \in [0, t_0]$ and $\psi_{k+1} = T\psi_k$ for all $k \geq 0$. By the Banach fixed-point theorem~\citep{Granas-2013-Fixed}, the Picard iterates converge to a unique fixed point of $T$ if $\XCal$ is an nonempty and complete metric space and $T$ is a contraction from $\XCal$ to $\XCal$.

\textit{First, we show that $\XCal$ is an nonempty and complete metric space.} Indeed, we define $d(x, x') = \max_{t \in [0, t_0]} \|x(t) - x'(t)\|$. It is easy to verify that $d$ is a metric and $(\XCal, d)$ is a complete metric space (see~\citet{Sutherland-2009-Introduction} for the details). In addition, $\XCal$ is nonempty since the constant function $x = x_0$ is one element.    

\textit{It remains to prove that $T$ is a contraction for some $t_0 < t_1$.} Indeed, $\Lambda_\theta(z)$ and $\nabla\Phi(z)$ are bounded for $\forall z \in \BB_\delta(x_0)$; that is, there exists $M_1>0$ such that $\max\{\Lambda_\theta(z), \|\nabla\Phi(z)\|\} \leq M_1$ for $\forall z \in \BB_\delta(x_0)$. By Proposition~\ref{Prop:AE-mapping-Key}, $\Lambda_\theta$ and $\sqrt{\Lambda_\theta}$ are continuous and locally Lipschitz over $\Omega$. Since $\BB_{\delta}(x_0) \subseteq \Omega$ is bounded, there exists $L_1>0$ such that, for any $x', x'' \in \BB_\delta(x_0)$, we have
\begin{equation}\label{inequality:LEU-second}
\max\{|\Lambda_\theta(x') - \Lambda_\theta(x'')|, |\sqrt{\Lambda_\theta}(x') - \sqrt{\Lambda_\theta}(x'')|\} \leq L_1\|x' - x''\|. 
\end{equation}
Note that $\Phi$ is twice continuously differentiable. Thus, there exists $L_2>0$ such that $\|\nabla\Phi(x') - \nabla\Phi(x'')\| \leq L_2\|x' - x''\|$ for $\forall x', x'' \in \BB_\delta(x_0)$. In addition, for any $t \in [0, t_0]$, we have $\|x(t)\| \leq \|x_0\| + \delta = M_2$. 

We now proceed to the main proof. By the triangle inequality, we have
\begin{align*}
\lefteqn{\|Tx'(t) - Tx''(t)\| \leq \underbrace{\int_0^t \|\Lambda_\theta(x'(s))\nabla\Phi(x'(s)) - \Lambda_\theta(x''(s))\nabla\Phi(x''(s))\| ds}_{\textbf{I}}} \\
& + \int_0^t \left\|\frac{\sqrt{\Lambda_\theta(x'(s))}}{\int_0^s \sqrt{\Lambda_\theta(x'(w))} dw + c}\left(\int_0^s \left(\sqrt{\Lambda_\theta(x'(w))}\left(\int_0^w \sqrt{\Lambda_\theta(x'(v))} \ dv + c\right)\right) \nabla \Phi(x'(w))dw\right) \right. \\
& \underbrace{\left. - \frac{\sqrt{\Lambda_\theta(x''(s))}}{\int_0^s \sqrt{\Lambda_\theta(x''(w))} dw + c}\left(\int_0^s \left(\sqrt{\Lambda_\theta(x''(w))}\left(\int_0^w \sqrt{\Lambda_\theta(x''(v))} \ dv + c\right)\right) \nabla \Phi(x''(w))dw\right)\right\| ds}_{\textbf{II}} \\
& + \underbrace{\int_0^t \left\|\frac{2\sqrt{\Lambda_\theta(x'(s))}}{\int_0^s \sqrt{\Lambda_\theta(x'(w))}dw + c}(x'(s) - v_0) - \frac{2\sqrt{\Lambda_\theta(x''(s))}}{\int_0^s \sqrt{\Lambda_\theta(x''(w))}dw + c}(x''(s) - v_0)\right\| ds}_{\textbf{III}}. 
\end{align*}
The key inequality for the subsequent analysis is as follows:
\begin{equation}\label{inequality:LEU-third}
\|a_1 b_1 - a_2 b_2\| \leq \|a_1\|\|b_1 - b_2\| + \|b_2\|\|a_1 - a_2\|. 
\end{equation}
First, by combining Eq.~\eqref{inequality:LEU-third} with $\max\{\Lambda_\theta(x(t)), \|\nabla\Phi(x(t))\|\} \leq M_1$, $\|\nabla\Phi(x') - \nabla\Phi(x'')\| \leq L_2\|x' - x''\|$ and Eq.~\eqref{inequality:LEU-second}, we obtain: 
\begin{equation*}
\textbf{I} \leq M_1(L_1 + L_2)t_0 d(x', x''). 
\end{equation*}
Second, we combine Eq.~\eqref{inequality:LEU-third} with $\sqrt{\Lambda_\theta(x(t))} \leq \sqrt{M_1}$, Eq.~\eqref{inequality:LEU-second} and $0 < s \leq t_0 < t_1 < 1$ to obtain:  
\begin{equation*}
\left\|\frac{\sqrt{\Lambda_\theta(x'(s))}}{\int_0^s \sqrt{\Lambda_\theta(x'(w))}dw + c} - \frac{\sqrt{\Lambda_\theta(x''(s))}}{\int_0^s \sqrt{\Lambda_\theta(x''(w))}dw + c}\right\| \leq \left(\frac{1}{c}+\frac{2\sqrt{M_1}}{c^2}\right)L_1d(x', x''). 
\end{equation*}
We also obtain by combining Eq.~\eqref{inequality:LEU-third} with $\max\{\Lambda_\theta(x(t)), \|\nabla\Phi(x(t))\|\} \leq M_1$, $\|\nabla\Phi(x') - \nabla\Phi(x'')\| \leq L_2\|x' - x''\|$, Eq.~\eqref{inequality:LEU-second} and $0 < w \leq s \leq t_0 < t_1 < 1$ that 
\begin{align*}
\lefteqn{\left\|\int_0^s \left(\sqrt{\Lambda_\theta(x'(w))}\left(\int_0^w \sqrt{\Lambda_\theta(x'(v))} \ dv + c\right)\right) \nabla \Phi(x'(w))dw \right.} \\
& \left. \ \ \ - \int_0^s \left(\sqrt{\Lambda_\theta(x''(w))}\left(\int_0^w \sqrt{\Lambda_\theta(x''(v))} \ dv + c\right)\right) \nabla \Phi(x''(w))dw\right\| \\
& \leq (M_1L_2 + c\sqrt{M_1}L_2 + 2(M_1)^{3/2}L_1 + cM_1L_1)d(x', x''). 
\end{align*}
In addition, by using $\max\{\Lambda_\theta(x(t)), \|\nabla\Phi(x(t))\|\} \leq M_1$ and $0 < w \leq s \leq t_0 < t_1 < 1$, we have
\begin{align*}
& \left\|\frac{\sqrt{\Lambda_\theta(x'(s))}}{\int_0^s \sqrt{\Lambda_\theta(x'(w))} \ dw + c}\right\| \leq \frac{\sqrt{M_1}}{c}, \\
& \left\|\int_0^s \left(\sqrt{\Lambda_\theta(x''(w))}\left(\int_0^w \sqrt{\Lambda_\theta(x''(v))} \ dv + c\right)\right) \nabla \Phi(x''(w))dw\right\| \leq (M_1)^2 + c(M_1)^{3/2}. 
\end{align*}
Putting these pieces together yields that  
\begin{equation*}
\textbf{II} \leq \left(\frac{2(M_1)^{5/2}L_1}{c^2} + \frac{(M_1)^{3/2}L_2 + 5(M_1)^2L_1}{c} + M_1L_2 + 2(M_1)^{3/2}L_1\right)t_0 d(x', x'').
\end{equation*}
Finally, by a similar argument, we have
\begin{equation*}
\textbf{III} \leq \left(\frac{2\sqrt{M_1} + 2(M_2+\|v_0\|)L_1}{c} + \frac{4\sqrt{M_1}(M_2+\|v_0\|)L_1}{c^2}\right)t_0d(x', x''). 
\end{equation*}
Combining the upper bounds for $\textbf{I}$, $\textbf{II}$ and $\textbf{III}$, we have 
\begin{equation*}
d(Tx', Tx'') = \max_{t \in [0, t_0]} \|Tx'(t) - Tx''(t)\| \leq \bar{M}t_0d(x', x''), 
\end{equation*}
where $\bar{M}$ is a constant that does not depend on $t_0$ (in fact it depends on $c$, $x_0$, $\delta$, $\Phi(\cdot)$ and $\Lambda_\theta(\cdot)$) and is defined as follows: 
\begin{eqnarray*}
\bar{M} & = & \frac{2((M_1)^2+2M_2+2\|v_0\|)\sqrt{M_1}L_1}{c^2} + \frac{2\sqrt{M_1} + (2M_2+2\|v_0\|+5(M_1)^2)L_1+(M_1)^{3/2}L_2}{c} \\
& & + 2M_1L_2 + (M_1+2(M_1)^{3/2})L_1. 
\end{eqnarray*}
Therefore, the mapping $T$ is a contraction if $t_0 \in (0, t_1]$ satisfies $t_0 \leq \frac{1}{2\bar{M}}$. This completes the proof. 
\end{proof}

\subsection{Discussion}\label{subsec:Discussion}
We compare the closed-loop control system in Eq.~\eqref{Sys:general} and Eq.~\eqref{Sys:choice-feedback} with four main classes of systems in the literature.  

\paragraph{Hessian-driven damping.} The formal introduction of Hessian-driven damping in optimization dates to~\citet{Alvarez-2002-Second}, with many subsequent developments; see, e.g.,~\citet{Attouch-2016-Fast}. The system studied in this literature takes the following form: 
\begin{equation*}
\ddot{x}(t) + \frac{\alpha}{t}\dot{x}(t) + \beta\nabla^2\Phi(x(t))\dot{x}(t) + \nabla\Phi(x(t)) = 0. 
\end{equation*}
In a Hilbert space setting and when $\alpha>3$, the literature has established the weak convergence of any solution trajectory to a global minimizer of $\Phi$ and the convergence rate of $o(1/t^2)$ in terms of objective function gap.
    
Recall also that~\citet{Shi-2018-Understanding} interpreted  Nesterov acceleration as the discretization of a high-resolution differential equation:
\begin{equation*}
\ddot{x}(t) + \frac{3}{t}\dot{x}(t) + \sqrt{s}\nabla^2\Phi(x(t))\dot{x}(t) + \left(1+\frac{3\sqrt{s}}{2t}\right)\nabla\Phi(x(t)) = 0,  
\end{equation*}  
and showed that this equation distinguishes between Polyak's heavy-ball method and Nesterov's accelerated gradient method. In the special case in which $c = 0$ and $p = 1$, our system in Eq.~\eqref{Sys:general} and Eq.~\eqref{Sys:choice-feedback} becomes 
\begin{equation}\label{Sys:specific-p-1}
\ddot{x}(t) + \frac{3}{t}\dot{x}(t) + \theta\nabla^2\Phi(x(t))\dot{x}(t) + \left(\theta + \frac{\theta}{t}\right)\nabla\Phi(x(t)) = 0. 
\end{equation}  
which also belongs to the class of high-resolution differential equations.  Moreover, for $c = 0$ and $p = 1$, our system can be studied within the recently-proposed framework of~\citet{Attouch-2020-Fast,Attouch-2021-Fast}; indeed, in this case $(\alpha, \beta, b)$ in~\citet[Theorem~2.1]{Attouch-2021-Fast} has an analytic form.  However, the choice of $(\alpha, \beta, b)$ in our general setting in Eq.~\eqref{Sys:choice-feedback}, for $p \geq 2$, does not have an analytic form and it is  difficult to verify whether $(\alpha, \beta, b)$ in this case satisfies their condition.

\paragraph{Newton and Levenberg-Marquardt regularized systems.} The precursor of this perspective was developed by~\citet{Alvarez-1998-Dynamical} in a variational characterization of general regularization algorithms. By constructing the regularization of the potential function $\Phi(\cdot, \epsilon)$ satisfying $\Phi(\cdot, \epsilon) \rightarrow \Phi$ as $\epsilon \rightarrow 0$, they studied the following system:
\begin{equation*}
\nabla^2\Phi(x(t), \epsilon(t))\dot{x}(t) + \dot{\epsilon}(t)\frac{\partial^2\Phi}{\partial\epsilon \partial x}(x(t), \epsilon(t)) + \nabla\Phi(x(t), \epsilon(t)) = 0.
\end{equation*}
Subsequently,~\citet{Attouch-2001-Second} and~\citet{Attouch-2011-Continuous} studied Newton dissipative and Levenberg-Marquardt regularized systems:
\begin{align*}
\textbf{(Newton)} \qquad & \ddot{x}(t) + \nabla^2\Phi(x(t))\dot{x}(t) + \nabla\Phi(x(t)) = 0. \\
\textbf{(Levenberg-Marquardt)} \qquad & \lambda(t)\dot{x}(t) + \nabla^2\Phi(x(t))\dot{x}(t) + \nabla\Phi(x(t)) = 0. 
\end{align*}
These systems have been shown to be well defined and stable with robust asymptotic behavior~\citep{Attouch-2011-Continuous,Attouch-2013-Global,Abbas-2014-Newton}, further motivating the study of the following inertial gradient system with constant damping and Hessian-driven damping~\citep{Alvarez-2002-Second}:      
\begin{equation*}
\ddot{x}(t) + \alpha\dot{x}(t) + \beta\nabla^2\Phi(x(t))\dot{x}(t) + \nabla\Phi(x(t)) = 0.  
\end{equation*}
This system attains strong asymptotic stabilization and fast convergence properties~\citep{Alvarez-2002-Second,Attouch-2012-Second} and can be extended to solve the monotone inclusion problems with theoretical guarantee~\citep{Attouch-2011-Continuous,Mainge-2013-First,Attouch-2013-Global,Abbas-2014-Newton,Attouch-2016-Dynamic,Attouch-2020-Newton,Attouch-2020-Continuous}. However, all of these systems are aimed at interpreting standard and regularized Newton algorithms and fail to model optimal acceleration even for the second-order algorithms in~\citet{Monteiro-2013-Accelerated}.  

Recently,~\citet{Attouch-2016-Dynamic} proposed a proximal Newton algorithm for solving monotone inclusions, which is motivated by a closed-loop control system without inertia. This algorithm attains a suboptimal convergence rate of $O(t^{-2})$ in terms of objective function gap.

\paragraph{Closed-loop control systems.} The closed-loop damping approach in~\citet{Attouch-2013-Global,Attouch-2016-Dynamic} closely resembles ours. In particular, they interpret various Newton-type methods as the discretization of the closed-loop control system without inertia and prove the existence and uniqueness of a solution as well as the convergence rate of the solution trajectory. There are, however, some significant differences between our work and theirs. In particular, the appearance of inertia is well known to make analysis much more challenging. Standard existence and uniqueness proofs based on the Cauchy-Schwarz theorem suffice to analyze the system of~\citet{Attouch-2013-Global,Attouch-2016-Dynamic} thanks to the lack of inertia, while Picard iterates and the Banach fixed-point theorem are necessary for our analysis. The construction of the Lyapunov function is also more difficult for the system with inertia.

This is an active research area and we refer the interested reader to a recent article of~\citet{Attouch-2020-Fast} for a comprehensive treatment of this topic. 

\paragraph{Continuous-time interpretation of high-order tensor algorithms.} There is comparatively little work on continuous-time perspectives on high-order tensor algorithms; indeed, we are aware of only~\citet{Wibisono-2016-Variational} and~\citet{Song-2021-Unified}. 

By appealing to a variational formulation,~\citet{Wibisono-2016-Variational} derived the following inertial gradient system with asymptotic vanishing damping:
\begin{equation}\label{Sys:WWJ}
\ddot{x}(t) + \frac{p+2}{t}\dot{x}(t) + C(p+1)^2t^{p-1}\nabla\Phi(x(t)) = 0. 
\end{equation}
Compared to our closed-loop control system, in Eq.~\eqref{Sys:general} and Eq.~\eqref{Sys:choice-feedback}, the system in Eq.~\eqref{Sys:WWJ} is an open-loop system without the algebra equation and does not contain Hessian-driven damping. These differences yield solution trajectories that only attain a suboptimal convergence rate of $O(t^{-(p+1)})$ in terms of objective function gap.   

Very recently,~\citet{Song-2021-Unified} have proposed and analyzed the following dynamics (we consider the Euclidean setting for simplicity):
\begin{equation*}
\left\{\begin{array}{ll}
& a(t)\dot{x}(t) = \dot{a}(t)(z(t) - x(t)) \\ 
& z(t) = \argmin_{x \in \br^d} \int_0^t \dot{a}(s)(\Phi(x(s)) + \langle\nabla\Phi(x(s)), x - x(s)\rangle) ds + \frac{1}{2}\|x - x_0\|^2.
\end{array}\right.
\end{equation*}
Solving the minimization problem yields $z(t) = x_0 - \int_0^t \dot{a}(s)\nabla\Phi(x(s))ds$. Substituting and rearranging yields:
\begin{equation}\label{Sys:SYM}
\ddot{x}(t) + \left(\frac{2\dot{a}(t)}{a(t)} - \frac{\ddot{a}(t)}{\dot{a}(t)}\right)\dot{x}(t) + \left(\frac{(\dot{a}(t))^2}{a(t)}\right)\nabla\Phi(x(t)) = 0.  
\end{equation}
Compared to our closed-loop control system, the system in~\eqref{Sys:SYM} is open-loop and lacks Hessian-driven damping. Moreover, $a(t)$ needs to be determined by hand and~\citet{Song-2021-Unified} do not establish existence or uniqueness of solutions.

\section{Lyapunov Function}\label{sec:Lyapunov}
In this section, we construct a Lyapunov function that allows us to prove existence and uniqueness of a global solution of our closed-loop control system and to analyze convergence rates. As we will see, an analysis of the rate of decrease of the Lyapunov function together with the algebraic equation permit the derivation of new convergence rates for both the objective function gap and the squared gradient norm. 

\subsection{Existence and uniqueness of a global solution}
Our main theorem on the existence and uniqueness of a global solution is summarized as follows. 
\begin{theorem}\label{Theorem:Global-Existence-Uniquess} 
Suppose that $\lambda$ is absolutely continuous on any finite bounded interval. Then the closed-loop control system in Eq.~\eqref{Sys:general} and Eq.~\eqref{Sys:choice-feedback} has a unique global solution, $(x, \lambda, a): [0, +\infty) \mapsto \br^d \times (0, +\infty) \times (0, +\infty)$. 
\end{theorem}
\begin{remark}
Intuitively, the feedback law $\lambda(\cdot)$, which we will show satisfies $\lambda(t) \rightarrow +\infty$ as $t \rightarrow +\infty$, links to the gradient norm $\|\nabla\Phi(x(\cdot))\|$ via the algebraic equation. Since we are interested in the worst-case convergence rate of solution trajectories, which corresponds to the worst-case iteration complexity of discrete-time algorithms, it is necessary that $\lambda$ does not dramatically change. In open-loop Levenberg-Marquardt systems,~\citet{Attouch-2011-Continuous} impose the same condition on the regularization parameters. In closed-loop control systems, however, $\lambda$ is not a given datum but an emergent component of the dynamics. Thus, it is preferable to prove that $\lambda$ satisfies this condition rather than assuming it, as done in~\citet[Theorem~5.2]{Attouch-2013-Global} and~\citet[Theorem~2.4]{Attouch-2016-Dynamic} for a closed-loop control system without inertia. The key step in their proof is to show that $\lambda(t) \leq \lambda(0)e^t$ locally by exploiting the specific structure of their system. This technical approach is, however, not applicable to our system due to the incorporation of the inertia term; see Section~\ref{subsec:Lyapunov-dis} for further discussion.
\end{remark}
Recall that the system in Eq.~\eqref{Sys:general} and Eq.~\eqref{Sys:choice-feedback} can be equivalently written as the first-order system in time and space, as in Eq.~\eqref{Sys:FO}. Accordingly, we define the following simple Lyapunov function:
\begin{equation}\label{Def:Lyapunov}
\ECal(t) = a(t)(\Phi(x(t)) - \Phi(x^\star)) + \frac{1}{2}\|v(t) - x^\star\|^2,  
\end{equation}
where $x^\star$ is a global optimal solution of $\Phi$. 
\begin{remark}
Note that the Lyapunov function~\eqref{Def:Lyapunov} is composed of a sum of the mixed energy $\frac{1}{2}\|v(t) - x^*\|$ and the potential energy $a(t)(\Phi(x(t)) - \Phi(x^*))$. This function is similar to Lyapunov functions developed for analyzing the convergence of Newton-like dynamics~\citep{Attouch-2011-Continuous,Attouch-2013-Global,Abbas-2014-Newton,Attouch-2016-Dynamic} and the inertial gradient system with asymptotic vanishing damping~\citep{Su-2016-Differential,Attouch-2016-Fast,Shi-2018-Understanding,Wilson-2021-Lyapunov}. Indeed,~\citet{Wilson-2021-Lyapunov} construct a unified time-dependent Lyapunov function using the Bregman divergence and showed that their approach is equivalent to Nesterov's estimate sequence technique in a number of cases, including quasi-monotone subgradient, accelerated gradient descent and conditional gradient. Our Lyapunov function differs from existing choices in that $v$ is not a standard momentum term depending on $\dot{x}$, but depends on $x$, $\lambda$ and $\nabla\Phi$; see Eq.~\eqref{Sys:FO}.
\end{remark}
We provide two technical lemmas that characterize the descent property of $\ECal$ and the boundedness of the local solution $(x, v): [0, t_0] \mapsto \br^d \times \br^d$.
\begin{lemma}\label{Lemma:Lyapunov-first}
Suppose that $(x, v, \lambda, a): [0, t_0] \mapsto \br^d \times \br^d \times (0, +\infty) \times (0, +\infty)$ is a local solution of the first-order system in Eq.~\eqref{Sys:FO}. Then, we have
\begin{equation*}
\frac{d\ECal(t)}{dt} \leq - a(t)\theta^{\frac{1}{p}}\|\nabla\Phi(x(t))\|^{\frac{p+1}{p}}, \quad \forall t \in  [0, t_0]. 
\end{equation*}
\end{lemma}
\begin{proof}
By the definition, we have
\begin{equation*}
\frac{d\ECal(t)}{dt} = \dot{a}(t)\Phi(x(t)) - \dot{a}(t)\Phi(x^\star) + \langle a(t)\dot{x}(t), \nabla\Phi(x(t))\rangle + \langle\dot{v}(t), v(t) - x^\star\rangle.  
\end{equation*}
In addition, we have $\langle\dot{v}(t), v(t) - x^\star\rangle = \langle\dot{v}(t), v(t) - x(t)\rangle + \langle\dot{v}(t), x(t) - x^\star\rangle$ and $\dot{v}(t) = -\dot{a}(t) \nabla \Phi(x(t))$. Putting these pieces together yields: 
\begin{align*}
\lefteqn{\frac{d\ECal(t)}{dt} = \underbrace{\dot{a}(t)(\Phi(x(t)) - \Phi(x^\star) - \langle \nabla\Phi(x(t)), x(t) - x^\star\rangle)}_{\textbf{I}}} \\
& + \underbrace{\langle a(t)\dot{x}(t), \nabla\Phi(x(t))\rangle + \dot{a}(t)\langle x(t) - v(t), \nabla\Phi(x(t))\rangle}_{\textbf{II}}.  
\end{align*}
By the convexity of $\Phi$, we have $\Phi(x(t)) - \Phi(x^\star) - \langle \nabla\Phi(x(t)), x(t) - x^\star\rangle \leq 0$. Since $\dot{a}(t) \geq 0$, we have $\textbf{I} \leq 0$. Furthermore, Eq.~\eqref{Sys:FO} implies that  
\begin{equation*}
\dot{x}(t) + \frac{\dot{a}(t)}{a(t)}(x(t) - v(t)) = - \lambda(t)\nabla\Phi(x(t)),  
\end{equation*}
which implies that 
\begin{equation*}
\textbf{II} = \langle a(t)\dot{x}(t) + \dot{a}(t)x(t) - \dot{a}(t)v(t), \nabla\Phi(x(t))\rangle = - \lambda(t)a(t)\|\nabla\Phi(x(t))\|^2. 
\end{equation*}
This together with the algebraic equation implies $\textbf{II} \leq -a(t)\theta^{\frac{1}{p}}\|\nabla\Phi(x(t))\|^{\frac{p+1}{p}}$. Putting all these pieces together yields the desired inequality. 
\end{proof}
In view of Lemma~\ref{Lemma:Lyapunov-first}, the key ingredient for analyzing the convergence rate in terms of the objective function gap and the squared gradient norm is a lower bound on $a(t)$. We summarize this result in the following lemma.   
\begin{lemma}\label{Lemma:Lyapunov-second}
Suppose that $(x, v, \lambda, a): [0, t_0] \mapsto \br^d \times \br^d \times (0, +\infty) \times (0, +\infty)$ is a local solution of the first-order system in Eq.~\eqref{Sys:FO}. Then, we have
\begin{equation*}
a(t) \geq \left(\frac{c}{2} + \left(\frac{\theta^{\frac{2}{3p+1}}}{(p+1)(\ECal(0))^{\frac{p-1}{3p+1}}}\right)^{\frac{3p+1}{4}} t^{\frac{3p+1}{4}}  \right)^2, \quad \forall t \in [0, t_0]. 
\end{equation*} 
\end{lemma}
\begin{proof}
For $p=1$, the feedback control law is given by $\lambda(t) = \theta$, for $\forall t \in [0, t_0]$, and
\begin{equation*}
a(t) = \left(\frac{c}{2} + \frac{\sqrt{\theta}t}{2}\right)^2 = \left(\frac{c}{2} + \left(\frac{\theta^{\frac{2}{3p+1}}}{(p+1)(\ECal(0))^{\frac{p-1}{3p+1}}}\right)^{\frac{3p+1}{4}} t^{\frac{3p+1}{4}}  \right)^2.
\end{equation*}
For $p \geq 2$, the algebraic equation implies that $\|\nabla\Phi(x(t))\| = (\frac{\theta^{1/p}}{\lambda(t)})^{\frac{p}{p-1}}$ since $\lambda(t) > 0$ for $\forall t \in [0, t_0]$. This together with Lemma~\ref{Lemma:Lyapunov-third} implies that
\begin{equation*}
\frac{d\ECal(t)}{dt} \leq - a(t)\theta^{\frac{1}{p}}\left\|\nabla\Phi(x(t))\right\|^{\frac{p+1}{p}} = - a(t)\theta^{\frac{2}{p-1}}[\lambda(t)]^{-\frac{p+1}{p-1}}. 
\end{equation*}
Since $\ECal(t) \geq 0$, we have
\begin{equation*}
\int_0^t a(s)\theta^{\frac{2}{p-1}}(\lambda(s))^{-\frac{p+1}{p-1}} ds \leq \ECal(0). 
\end{equation*}
By the H\"{o}lder inequality, we have
\begin{align*}
\lefteqn{\int_0^t (a(s))^{\frac{p-1}{3p+1}} ds = \int_0^t (a(s)(\lambda(s))^{-\frac{p+1}{p-1}})^{\frac{p-1}{3p+1}} (\lambda(s))^{\frac{p+1}{3p+1}} ds} \\ 
& \leq \left(\int_0^t a(s)(\lambda(s))^{-\frac{p+1}{p-1}} ds\right)^{\frac{p-1}{3p+1}}\left(\int_0^t \sqrt{\lambda(s)} ds\right)^{\frac{2p+2}{3p+1}}.
\end{align*}
Combining these results with the definition of $a$ yields: 
\begin{align*}
\lefteqn{\int_0^t (a(s))^{\frac{p-1}{3p+1}} ds \leq \theta^{-\frac{2}{3p+1}}(\ECal(0))^{\frac{p-1}{3p+1}}\left(\int_0^t \sqrt{\lambda(s)} ds\right)^{\frac{2p+2}{3p+1}}} \\ 
& \leq \theta^{-\frac{2}{3p+1}}(\ECal(0))^{\frac{p-1}{3p+1}}(2\sqrt{a(t)} - c)^{\frac{2p+2}{3p+1}} \leq 2\theta^{-\frac{2}{3p+1}}(\ECal(0))^{\frac{p-1}{3p+1}}\left(\sqrt{a(t)} - \frac{c}{2}\right)^{\frac{2p+2}{3p+1}}.  
\end{align*}
Since $a(t)$ is nonnegative and nondecreasing with $\sqrt{a(0)} = \frac{c}{2}$, we have
\begin{equation}\label{inequality:Lyapunov-second-main}
\int_0^t \left(\sqrt{a(s)} - \frac{c}{2}\right)^{\frac{2p-2}{3p+1}} ds \leq 2\theta^{-\frac{2}{3p+1}}(\ECal(0))^{\frac{p-1}{3p+1}}\left(\sqrt{a(t)} - \frac{c}{2}\right)^{\frac{2p+2}{3p+1}}.  
\end{equation}
The remaining steps in the proof are based on the Bihari-LaSalle inequality~\citep{Lasalle-1949-Uniqueness,Bihari-1956-Generalization}. In particular, we denote $y(\cdot)$ by $y(t) = \int_0^t (\sqrt{a(s)} - \frac{c}{2})^{\frac{2p-2}{3p+1}} ds$. Then, $y(0) = 0$ and Eq.~\eqref{inequality:Lyapunov-second-main} implies
\begin{equation*}
y(t) \leq 2\theta^{-\frac{2}{3p+1}}(\ECal(0))^{\frac{p-1}{3p+1}}(\dot{y}(t))^{\frac{p+1}{p-1}}.
\end{equation*}
This implies
\begin{equation*}
\dot{y}(t) \geq \left(\frac{y(t)}{2\theta^{-\frac{2}{3p+1}}(\ECal(0))^{\frac{p-1}{3p+1}}}\right)^{\frac{p-1}{p+1}} \Longrightarrow \frac{\dot{y}(t)}{(y(t))^{\frac{p-1}{p+1}}} \geq \left(\frac{1}{2\theta^{-\frac{2}{3p+1}}(\ECal(0))^{\frac{p-1}{3p+1}}}\right)^{\frac{p-1}{p+1}}. 
\end{equation*}
Integrating this inequality over $[0, t]$ yields: 
\begin{equation*}
(y(t))^{\frac{2}{p+1}} \geq \frac{2}{p+1}\left(\frac{1}{2\theta^{-\frac{2}{3p+1}}(\ECal(0))^{\frac{p-1}{3p+1}}}\right)^{\frac{p-1}{p+1}} t. 
\end{equation*}
Equivalently, by the definition of $y(t)$, we have
\begin{equation*}
\int_0^t \left(\sqrt{a(s)} - \frac{c}{2}\right)^{\frac{2p-2}{3p+1}} ds \geq \left(\frac{2}{p+1}\right)^{\frac{p+1}{2}}\left(\frac{1}{2\theta^{-\frac{2}{3p+1}}(\ECal(0))^{\frac{p-1}{3p+1}}}\right)^{\frac{p-1}{2}} t^{\frac{p+1}{2}}. 
\end{equation*}
This together with Eq.~\eqref{inequality:Lyapunov-second-main} yields that  
\begin{align*}
\sqrt{a(t)} & \geq \frac{c}{2} + \left(\frac{1}{2\theta^{-\frac{2}{3p+1}}(\ECal(0))^{\frac{p-1}{3p+1}}}\int_0^t \left(\sqrt{a(s)} - \frac{c}{2}\right)^{\frac{2p-2}{3p+1}} ds\right)^{\frac{3p+1}{2p+2}} \\
& \geq \frac{c}{2} + \left(\frac{\theta^{\frac{2}{3p+1}}}{(p+1)(\ECal(0))^{\frac{p-1}{3p+1}}}\right)^{\frac{3p+1}{4}} t^{\frac{3p+1}{4}}.   
\end{align*}
This completes the proof. 
\end{proof}
\begin{lemma}\label{Lemma:Lyapunov-third}
Suppose that $(x, v, \lambda, a): [0, t_0] \mapsto \br^d \times \br^d \times (0, +\infty) \times (0, +\infty)$ is a local solution of the first-order system in Eq.~\eqref{Sys:FO}. Then, $(x(\cdot), v(\cdot))$ is bounded over the interval $[0, t_0]$ and the upper bound only depends on the initial condition.  
\end{lemma}
\begin{proof}
By Lemma~\ref{Lemma:Lyapunov-first}, the function $\ECal$ is nonnegative and nonincreasing on the interval $[0, t_0]$. This implies that, for any $t \in [0, t_0]$, we have 
\begin{equation*}
\frac{1}{2}\|v(t) - x^\star\|^2 \leq a(t)(\Phi(x(t)) - \Phi(x^\star)) + \frac{1}{2}\|v(t) - x^\star\|^2 \leq \ECal(0). 
\end{equation*}
Therefore, $v(\cdot)$ is bounded on the interval $[0, t_0]$ and the upper bound only depends on the initial condition. Furthermore, we have 
\begin{equation*}
a(t)(x(t) - x^\star) - a(0)(x_0 - x^\star) = \int_0^t (\dot{a}(s)(x(s) - x^\star) + a(s)\dot{x}(s)) ds. 
\end{equation*}
Using the triangle inequality and $a(0) = c^2$, we have
\begin{align*}
\lefteqn{\|a(t)(x(t) - x^\star)\| \leq c^2\|x_0 - x^\star\| + \int_0^t \|a(s)\dot{x}(s) + \dot{a}(t)x(s) - \dot{a}(s)x^\star\| ds} \\
& \overset{\textnormal{Eq.~\eqref{Sys:FO}}}{\leq} c^2\|x_0 - x^\star\| + \int_0^t \|\dot{a}(s)v(s) - \dot{a}(s)x^\star\| ds + \int_0^t \|\lambda(s)a(s)\nabla\Phi(x(s))\| ds.
\end{align*}
Note that $\|v(t) - x^\star\| \leq \sqrt{2\ECal(0)}$ is proved for all $t \in [0, t_0]$ and $a(t)$ is monotonically increasing with $a(0)=c^2$. Thus, the following inequality holds:
\begin{eqnarray}\label{inequality:Lyapunov-third-1}
\|x(t) - x^\star\| & \leq & \frac{c^2\|x_0 - x^\star\| + (a(t) - c^2)\sqrt{2\ECal(0)} + \int_0^t \lambda(s)a(s)\|\nabla\Phi(x(s))\| ds}{a(t)} \\ 
& \leq & \|x_0 - x^\star\| + \sqrt{2\ECal(0)} + \frac{1}{a(t)}\int_0^t \lambda(s)a(s)\|\nabla\Phi(x(s))\| ds. \nonumber
\end{eqnarray}
The algebra equation implies that $\lambda(t)\|\nabla\Phi(x(t))\| = \theta^{\frac{1}{p}}\|\nabla\Phi(x(t))\|^{\frac{1}{p}}$. Thus, we have 
\begin{equation*}
\int_0^t \lambda(s)a(s)\|\nabla\Phi(x(s))\| ds = \theta^{\frac{1}{p}}\left(\int_0^t a(s)\|\nabla\Phi(x(s))\|^{\frac{1}{p}} ds\right). 
\end{equation*}
By the H\"{o}lder inequality, Lemma~\ref{Lemma:Lyapunov-first} and using the fact that $a(t)$ is monotonically increasing, we have 
\begin{eqnarray*}
\lefteqn{\int_0^t a(s)\|\nabla\Phi(x(s))\|^{\frac{1}{p}} ds = \int_0^t (a(s))^{\frac{p}{p+1}}(a(s)\|\nabla\Phi(x(s))\|^{\frac{p+1}{p}})^{\frac{1}{p+1}} ds} \\ 
& \leq & \left(\int_0^t a(s) ds\right)^{\frac{p}{p+1}}\left(\int_0^t a(s)\|\nabla\Phi(x(s))\|^{\frac{p+1}{p}} ds\right)^{\frac{1}{p+1}} \ \leq \ \theta^{-\frac{1}{p(p+1)}}(\ECal(0))^{\frac{1}{p+1}}\left(\int_0^t a(s) ds\right)^{\frac{p}{p+1}} \\
& \leq & \theta^{-\frac{1}{p(p+1)}}(\ECal(0))^{\frac{1}{p+1}}(ta(t))^{\frac{p}{p+1}}. 
\end{eqnarray*}
Putting these pieces together yields 
\begin{equation}\label{inequality:Lyapunov-third-2}
\frac{1}{a(t)}\int_0^t \lambda(s)a(s)\|\nabla\Phi(x(s))\| ds \leq \left(\frac{\theta\ECal(0) t^p}{a(t)}\right)^{\frac{1}{p+1}}. 
\end{equation}
Combining Eq.~\eqref{inequality:Lyapunov-third-1} with Eq.~\eqref{inequality:Lyapunov-third-2} and Lemma~\ref{Lemma:Lyapunov-second} yields that $\|x(t) - x^\star\| \leq \|x_0 - x^\star\| +\sqrt{2\ECal(0)} + C$ where $C > 0$ depends on $\ECal(0)$, $p$, $c$ and $\theta$. Therefore, $x(t)$ is bounded on the interval $[0, t_0]$ and the upper bound only depends on the initial condition. This completes the proof. 
\end{proof}
\paragraph{Proof of Theorem~\ref{Theorem:Global-Existence-Uniquess}:} We are ready to prove our main result on the existence and uniqueness of a global solution. In particular, let us consider a maximal solution of the closed-loop control system in Eq.~\eqref{Sys:general} and Eq.~\eqref{Sys:choice-feedback}: 
\begin{equation*}
(x, \lambda, a): [0, T_{\max}) \mapsto \Omega \times (0, +\infty) \times (0, +\infty). 
\end{equation*}
The existence of a maximal solution follows from a classical argument relying on the existence and uniqueness of a local solution (see Theorem~\ref{Theorem:Local-Existence-Uniquess}). 

It remains to show that the maximal solution is a global solution; that is, $T_{\max} = +\infty$, if $\lambda$ is absolutely continuous on any finite bounded interval. Indeed, the property of $\lambda$ guarantees that $\lambda(\cdot)$ is bounded on the interval $[0, T_{\max})$. By Lemma~\ref{Lemma:Lyapunov-third} and the equivalence between the closed-loop control system in Eq.~\eqref{Sys:general} and Eq.~\eqref{Sys:choice-feedback} and the first-order system in Eq.~\eqref{Sys:FO}, the solution trajectory $x(\cdot)$ is bounded on the interval $[0, T_{\max})$ and the upper bound only depends on the initial condition. This implies that $\dot{x}(\cdot)$ is also bounded on the interval $[0, T_{\max})$ by considering the system in the autonomous form of Eq.~\eqref{Sys:autonomous} and~\eqref{Sys:vector-field}. Putting these pieces together yields that $x(\cdot)$ is Lipschitz continuous on $[0, T_{\max})$ and there exists $\bar{x} = \lim_{t \rightarrow T_{\max}} x(t)$. 

If $T_{\max} < +\infty$, the absolute continuity of $\lambda$ on any finite bounded interval implies that $\lambda(\cdot)$ is bounded on $[0, T_{\max}]$. This together with the algebraic equation implies that $\bar{x} \in \Omega$. However, by Theorem~\ref{Theorem:Local-Existence-Uniquess} with initial data $\bar{x}$, we can extend the solution to a strictly larger interval which contradicts the maximality of the aforementioned solution. This completes the proof. 

\subsection{Rate of convergence}
We establish a convergence rate for a global solution of the closed-loop control system in Eq.~\eqref{Sys:general} and Eq.~\eqref{Sys:choice-feedback}. 
\begin{theorem}\label{Theorem:Trajectory-Convergence-Rate}
Suppose that $(x, \lambda, a): [0, +\infty) \mapsto \br^d \times (0, +\infty) \times (0, +\infty)$ is a global solution of the closed-loop control system in Eq.~\eqref{Sys:general} and Eq.~\eqref{Sys:choice-feedback}. Then, the objective function gap satisfies
\begin{equation*}
\Phi(x(t)) - \Phi(x^\star) = O(t^{-\frac{3p+1}{2}}).  
\end{equation*}
and the squared gradient norm satisfies
\begin{equation*}
\inf_{0 \leq s \leq t} \|\nabla \Phi(x(s))\|^2 = O(t^{-3p}).  
\end{equation*}
\end{theorem}
\begin{remark}
This theorem shows that the convergence rate is $O(t^{-(3p+1)/2})$ in terms of objective function gap and $O(t^{-3p})$ in terms of squared gradient norm. Note that the former result does not imply the latter result but only gives a rate of $O(t^{-(3p+1)/2})$ for the squared gradient norm minimization even when $\Phi \in \FCal_\ell^1(\br^d)$ is assumed with $\|\nabla\Phi(x(t))\|^2 \leq 2\ell(\Phi(x(t)) - \Phi(x^\star))$. In fact, the squared gradient norm minimization is generally of independent interest~\citep{Nesterov-2012-Make,Shi-2018-Understanding,Grapiglia-2020-Stationary} and its analysis involves different techniques.
\end{remark}
The following two lemmas are a global version of Lemmas~\ref{Lemma:Lyapunov-first} and~\ref{Lemma:Lyapunov-second} and the proofs are the same. Thus, we only state the result. 
\begin{lemma}\label{Lemma:Lyapunov-fourth}
Suppose that $(x, v, \lambda, a): [0, +\infty) \mapsto \br^d \times \br^d \times (0, +\infty) \times (0, +\infty)$ is a global solution of the first-order system in Eq.~\eqref{Sys:FO}. Then, we have
\begin{equation*}
\frac{d\ECal(t)}{dt} \leq -a(t)\theta^{\frac{1}{p}}\|\nabla\Phi(x(t))\|^{\frac{p+1}{p}}. 
\end{equation*}
\end{lemma}
\begin{lemma}\label{Lemma:Lyapunov-fifth}
Suppose that $(x, v, \lambda, a): [0, +\infty) \mapsto \br^d \times \br^d \times (0, +\infty) \times (0, +\infty)$ is a global solution of the first-order system in Eq.~\eqref{Sys:FO}. Then, we have
\begin{equation*}
a(t) \geq \left(\frac{c}{2} + \left(\frac{\theta^{\frac{2}{3p+1}}}{(p+1)(\ECal(0))^{\frac{p-1}{3p+1}}}\right)^{\frac{3p+1}{4}} t^{\frac{3p+1}{4}}  \right)^2. 
\end{equation*} 
\end{lemma}
\paragraph{Proof of Theorem~\ref{Theorem:Trajectory-Convergence-Rate}:} Since the first-order system in Eq.~\eqref{Sys:FO} is equivalent to the closed-loop control system in Eq.~\eqref{Sys:general} and Eq.~\eqref{Sys:choice-feedback}, $(x, \lambda, a): [0, +\infty) \rightarrow \br^d \times (0, +\infty) \times (0, +\infty)$ is a global solution of the latter system with $x(0) = x_0 \in \Omega$. By Lemma~\ref{Lemma:Lyapunov-fourth}, we have $\ECal(t) \leq \ECal(0)$ for $\forall t \geq 0$; that is,   
\begin{equation*}
a(t)(\Phi(x(t)) - \Phi(x^\star)) + \frac{1}{2}\|v(t) - x^\star\|^2 \leq \ECal(0). 
\end{equation*}
Since $(x(0), v(0)) = (x_0, v_0)$ and $\|v(t) - x^\star\| \geq 0$, we have $a(t)(\Phi(x(t)) - \Phi(x^\star)) \leq \ECal(0)$. By Lemma~\ref{Lemma:Lyapunov-fifth}, we have 
\begin{equation*}
\Phi(x(t)) - \Phi(x^\star) \leq \ECal(0)\left(\frac{c}{2} + \left(\frac{\theta^{\frac{2}{3p+1}}}{(p+1)(\ECal(0))^{\frac{p-1}{3p+1}}}\right)^{\frac{3p+1}{4}} t^{\frac{3p+1}{4}}\right)^{-2} = O(t^{-\frac{3p+1}{2}}). 
\end{equation*}
Using Lemma~\ref{Lemma:Lyapunov-fourth} and the fact that $\ECal(t) \geq 0$ for $\forall t \in [0, +\infty)$, we have
\begin{equation*}
\int_0^t a(s)\theta^{\frac{1}{p}}\|\nabla\Phi(x(s))\|^{\frac{p+1}{p}} ds \leq \ECal(0),
\end{equation*}
which implies
\begin{equation*}
\left(\inf_{0 \leq s \leq t} \|\nabla \Phi(x(s))\|^{\frac{p+1}{p}}\right)\left(\int_0^t a(s) ds\right) \leq \theta^{-\frac{1}{p}}\ECal(0). 
\end{equation*}
By Lemma~\ref{Lemma:Lyapunov-fifth}, we have
\begin{equation*}
\int_0^t a(s) ds \geq \int_0^t \left(\frac{c}{2} + \left(\frac{\theta^{\frac{2}{3p+1}}}{(p+1)(\ECal(0))^{\frac{p-1}{3p+1}}}\right)^{\frac{3p+1}{4}} s^{\frac{3p+1}{4}}\right)^2 ds. 
\end{equation*}
In addition, $\inf_{0 \leq s \leq t} \|\nabla \Phi(x(s))\|^{\frac{p+1}{p}} = (\inf_{0 \leq s \leq t} \|\nabla \Phi(x(s))\|^2)^{\frac{p+1}{2p}}$. Putting these pieces together yields 
\begin{equation*}
\inf_{0 \leq s \leq t} \|\nabla \Phi(x(s))\|^2 \leq \left(\frac{\theta^{-\frac{1}{p}}\ECal(0)}{\int_0^t (\frac{c}{2} + (\frac{\theta^{\frac{2}{3p+1}}}{(p+1)(\ECal(0))^{\frac{p-1}{3p+1}}})^{\frac{3p+1}{4}} s^{\frac{3p+1}{4}})^2 ds}\right)^{\frac{2p}{p+1}} = O(t^{-3p}).  
\end{equation*}
This completes the proof. 

\subsection{Discussion}\label{subsec:Lyapunov-dis}
It is useful to compare our approach to approaches based on time scaling~\citep{Attouch-2019-Fast,Attouch-2019-Time,Attouch-2021-Fast,Attouch-2021-ADMM} and quasi-gradient methods~\citep{Begout-2015-Damped,Attouch-2020-Fast}.

\paragraph{Regularity condition.} Why is proving the existence and uniqueness of a global solution of the closed-loop control system in Eq.~\eqref{Sys:general} and Eq.~\eqref{Sys:choice-feedback} hard without the regularity condition? Our system differs from the existing systems in three respects: (i) the appearance of both $\ddot{x}$ and $\dot{x}$; (ii) the algebraic equation that links $\lambda$ and $\nabla\Phi(x)$; and (iii) the evolution dynamics depends on $\lambda$ via $a$ and $\dot{a}$. From a technical point of view, the combination of these features makes it challenging to control a lower bound on gradient norm $\|\nabla\Phi(x(\cdot))\|$ or an upper bound on the feedback control $\lambda(\cdot)$ on the local interval. In sharp contrast, $\|\nabla\Phi(x(t))\| \geq \|\nabla\Phi(x(0))\|e^{-t}$ or $\lambda(t) \leq \lambda(0)e^t$ can readily be derived for the Levenberg-Marquardt regularized system in~\citet[Corollary~3.3]{Attouch-2011-Continuous} and even the closed-loop control systems without inertia in~\citet[Theorem~5.2]{Attouch-2013-Global} and~\citet[Theorem~2.4]{Attouch-2016-Dynamic}. Thus, we can not exclude the case of $\lambda(t) \rightarrow +\infty$ on the bounded interval without the regularity condition and we accordingly fail to establish global existence and uniqueness. We consider it an interesting open problem to derive the regularity condition rather than imposing it as an assumption. 

\paragraph{Infinite-dimensional setting.} It is promising to study our system using the techniques developed by~\citet{Attouch-2016-Fast} for an infinite-dimensional setting. Our convergence analysis can in fact be extended directly, yielding the same rate of $O(1/t^{(3p+1)/2})$ in terms of objective function gap and $O(1/t^{3p})$ in terms of squared gradient norm in the Hilbert-space setting. However, the weak convergence of the solution trajectories is another matter. Note that~\citet{Attouch-2016-Fast} studied the following open-loop system with the parameters $(\alpha, \beta)$:
\begin{equation*}
\ddot{x}(t) + \frac{\alpha}{t}\dot{x}(t) + \beta\nabla^2\Phi(x(t))\dot{x}(t) + \nabla\Phi(x(t)) = 0.  
\end{equation*} 
The condition $\alpha>3$ is crucial for proving weak convergence of solution trajectories and establishing strong convergence in various practical situations. Indeed, the convergence of the solution trajectory has not been established so far when $\alpha=3$ (except in the one-dimensional case with $\beta=0$; see~\citet{Attouch-2019-Rate} for the reference). Unfortunately, when $c=0$ and $p=1$, the closed-loop control system in Eq.~\eqref{Sys:general} and Eq.~\eqref{Sys:choice-feedback} becomes
\begin{equation*}
\ddot{x}(t) + \frac{3}{t}\dot{x}(t) + \theta\nabla^2\Phi(x(t))\dot{x}(t) + \left(\theta + \frac{\theta}{t}\right)\nabla\Phi(x(t)) = 0. 
\end{equation*}  
The asymptotic damping coefficient $\frac{3}{t}$ does not satisfy the aforementioned condition in~\citet{Attouch-2016-Fast}, leaving doubt as to whether weak convergence holds true for the closed-loop control system in Eq.~\eqref{Sys:general} and Eq.~\eqref{Sys:choice-feedback}.

\paragraph{Time scaling.} In the context of non-autonomous dissipative systems, time scaling is a simple yet universally powerful tool to accelerate the convergence of solution trajectories~\citep{Attouch-2019-Fast,Attouch-2019-Time,Attouch-2021-Fast,Attouch-2021-ADMM}. Considering the general inertial gradient system in Eq.~\eqref{Sys:general}:
\begin{equation*}
\ddot{x}(t) + \alpha(t)\dot{x}(t) + \beta(t)\nabla^2\Phi(x(t))\dot{x}(t) + b(t)\nabla\Phi(x(t)) = 0, 
\end{equation*}  
the effect of time scaling is characterized by the coefficient parameter $b(t)$ which comes in as a factor of $\nabla\Phi(x(t))$. In~\citet{Attouch-2019-Fast,Attouch-2019-Time}, the authors conducted an in-depth study of the convergence of this above system without Hessian-driven damping ($\beta=0$). For the case $\alpha(t) = \frac{\alpha}{t}$, the convergence rate turns out to be $O(\frac{1}{t^2 b(t)})$ under certain conditions on the scalar $\alpha$ and $b(\cdot)$. Thus, a clear improvement can be achieved by taking $b(t) \rightarrow +\infty$. This demonstrates the power and potential of time scaling, as further evidenced by recent work on systems with Hessian damping~\citep{Attouch-2021-Fast} and other systems which are associated with the augmented Lagrangian formulation of the affine constrained convex minimization problem~\citep{Attouch-2021-ADMM}.

Comparing to our closed-loop damping approach, the time scaling technique is based on an open-loop control regime, and indeed $b(t)$ is chosen by hand. In contrast, $\lambda(t)$ in our system is determined by the gradient of $\nabla \Phi(x(t))$ via the algebraic equation, and the evolution dynamics depend on $\lambda$ via $a$ and $\dot{a}$. The time scaling methodology accordingly does not capture the continuous-time interpretation of optimal acceleration in high-order optimization~\citep{Monteiro-2013-Accelerated,Gasnikov-2019-Optimal,Jiang-2019-Optimal,Bubeck-2019-Near}. In contrast, our algebraic equation provides a rigorous justification for the large-step condition in the existing algorithms~\citep{Monteiro-2013-Accelerated,Gasnikov-2019-Optimal,Jiang-2019-Optimal,Bubeck-2019-Near} when $p \geq 2$ and demonstrates the fundamental role that the feedback control plays in optimal acceleration, a role clarified by the continuous-time perspective.      

\paragraph{Quasi-gradient approach and Kurdyka-Lojasiewicz (KL) theory.} The quasi-gradient approach to inertial gradient systems were developed in~\citet{Begout-2015-Damped} and recently applied by~\citet{Attouch-2020-Fast} to analyze inertial dynamics with closed-loop control of the velocity. Recall that a vector field $F$ is called a quasi-gradient for a function $E$ if it has the same singular point as $E$ and if the angle between the field $F$ and the gradient $\nabla E$ remains acute and bounded away from $\frac{\pi}{2}$ (see the references~\citep{Huang-2006-Gradient,Chergui-2008-Convergence,Chill-2010-Gradient,Barta-2012-Every,Barta-2016-Convergence} for further geometrical interpretation). 

Recent results in~\citet[Theorem~3.2]{Begout-2015-Damped} and~\citet[Theorem~7.2]{Attouch-2020-Fast} have suggested that the convergence properties for the bounded trajectories of quasi-gradient systems have been established if the function $E$ is KL~\citep{Kurdyka-1998-Gradients,Bolte-2010-Characterizations}. In~\citet{Attouch-2020-Fast}, the authors considered two closed-loop velocity control systems with a damping potential $\phi$: 
\begin{align}
& \ddot{x}(t) + \nabla\phi(\dot{x}(t)) + \nabla\Phi(x(t)) = 0. \label{Sys:QG-first} \\
& \ddot{x}(t) + \nabla\phi(\dot{x}(t)) + \beta\nabla^2\Phi(x(t))\dot{x}(t) + \nabla\Phi(x(t)) = 0. \label{Sys:QG-second}
\end{align}
They proposed to use the Hamiltonian formulation of these systems and accordingly defined a function $E_\lambda$ for $(x, v) =  (x, \dot{x}(t))$ by 
\begin{equation*}
E_\eta(x, v) := \frac{1}{2}\|v\|^2 + \Phi(x) + \eta\langle\nabla\Phi(x), v\rangle.  
\end{equation*}
If $\phi$ satisfies some certain growth conditions (see~\citet[Theorem~7.3 and~9.2]{Attouch-2020-Fast}), the systems in Eq.~\eqref{Sys:QG-first} and Eq.~\eqref{Sys:QG-second} both have a quasi-gradient structure for $E_\eta$ for sufficiently small $\eta > 0$. This provides an elegant framework for analyzing the convergence properties of the systems in the form of Eq.~\eqref{Sys:QG-first} and Eq.~\eqref{Sys:QG-second} with specific damping potentials.  

Why is analyzing our system hard using the quasi-gradient approach? Our system differs from the systems in Eq.~\eqref{Sys:QG-first} and Eq.~\eqref{Sys:QG-second} in two aspects: (i) the closed-loop control law is designed for the gradient of $\Phi$ rather than the velocity $\dot{x}$; (ii) the damping coefficients are time dependent, depending on $\lambda$ via $a$ and $\dot{a}$, and do not have an analytic form when $p \geq 2$. Considering the first-order systems in Eq.~\eqref{Sys:FO} and Eq.~\eqref{Sys:FO-inverse}, we find that $F$ is a time-dependent vector field which can not be tackled by the current quasi-gradient approach. We consider it an interesting open problem to develop a quasi-gradient approach for analyzing our system. 

\section{Implicit Time Discretization and Optimal Acceleration}\label{sec:optimal-algorithm}
In this section, we propose two conceptual algorithmic frameworks that arise via implicit time discretization of the closed-loop  system in Eq.~\eqref{Sys:FO} and Eq.~\eqref{Sys:FO-inverse}. Our approach demonstrates the importance of the large-step condition~\citep{Monteiro-2013-Accelerated} for optimal acceleration, interpreting it as the discretization of the algebraic equation. This allows us to further clarify why this condition is unnecessary for first-order optimization algorithms in the case of $p=1$ (the algebraic equation disappears). With an approximate tensor subroutine~\citep{Nesterov-2019-Implementable}, we derive two class of $p$-th order tensor algorithms, one of which recovers existing optimal $p$-th order tensor algorithms~\citep{Gasnikov-2019-Optimal,Jiang-2019-Optimal,Bubeck-2019-Near} and the other of which leads to a new optimal $p$-th order tensor algorithm. 

\subsection{Conceptual algorithmic frameworks}\label{subsec:framework}
We study two conceptual algorithmic frameworks which are derived by implicit time discretization of Eq.~\eqref{Sys:FO} with $c = 0$ and Eq.~\eqref{Sys:FO-inverse} with $c = 2$. 
\paragraph{First algorithmic framework.} By the definition of $a(t)$, we have $(\dot{a}(t))^2 = \lambda(t)a(t)$ and $a(0)=0$. This implies an equivalent formulation of the first-order system in Eq.~\eqref{Sys:FO} with $c = 0$ as follows,  
\begin{align*}
& \left\{\begin{array}{ll}
& \dot{v}(t) + \dot{a}(t)\nabla \Phi(x(t)) = 0 \\ 
& \dot{x}(t) + \frac{\dot{a}(t)}{a(t)}(x(t) - v(t)) + \frac{(\dot{a}(t))^2}{a(t)}\nabla\Phi(x(t)) = 0 \\
& a(t) = \frac{1}{4}(\int_0^t \sqrt{\lambda(s)} ds)^2 \\
& (\lambda(t))^p\|\nabla\Phi(x(t))\|^{p-1} = \theta \\ 
& (x(0), v(0)) = (x_0, v_0)    
\end{array}\right. \\
\Longleftrightarrow & 
\left\{\begin{array}{ll}
& \dot{v}(t) + \dot{a}(t)\nabla \Phi(x(t)) = 0 \\ 
& a(t)\dot{x}(t) + \dot{a}(t)(x(t) - v(t)) + \lambda(t)a(t)\nabla\Phi(x(t)) = 0 \\
& (\dot{a}(t))^2 = \lambda(t)a(t) \\
& (\lambda(t))^p\|\nabla\Phi(x(t))\|^{p-1} = \theta \\ 
& (x(0), v(0), a(0)) = (x_0, v_0, 0).   
\end{array}\right.  
\end{align*}
We define discrete-time sequences, $\{(x_k, v_k, \lambda_k, a_k, A_k)\}_{k \geq 0}$, that correspondx to the continuous-time sequences $\{(x(t), v(t), \lambda(t), \dot{a}(t), a(t))\}_{t \geq 0}$. By implicit time discretization, we have
\begin{equation}\label{DSys:FO}
\left\{\begin{array}{ll}
& v_{k+1} - v_k + a_{k+1}\nabla\Phi(x_{k+1}) = 0 \\ 
& A_{k+1}(x_{k+1} - x_k) + a_{k+1}(x_k - v_k) + \lambda_{k+1}A_{k+1}\nabla\Phi(x_{k+1}) = 0 \\
& (a_{k+1})^2 = \lambda_{k+1}(A_k + a_{k+1}), \ a_{k+1} = A_{k+1} - A_k, \ a_0 = 0 \\
& (\lambda_{k+1})^p\|\nabla\Phi(x_{k+1})\|^{p-1} = \theta.   
\end{array}\right.
\end{equation}
By introducing a new variable $\tilde{v}_k = \frac{A_k}{A_k + a_{k+1}}x_k + \frac{a_{k+1}}{A_k + a_{k+1}}v_k$, the second and fourth lines of Eq.~\eqref{DSys:FO} can be equivalently reformulated as follows: 
\begin{equation*}
\lambda_{k+1}\nabla\Phi(x_{k+1}) + x_{k+1} - \tilde{v}_k = 0, \qquad \lambda_{k+1}\|x_{k+1} - \tilde{v}_k\|^{p-1} = \theta. 
\end{equation*}
We propose to solve these two equations inexactly and replace $\nabla \Phi(x_{k+1})$ by a sufficiently accurate approximation in the first line of Eq.~\eqref{DSys:FO}. In particular, the first equation can be equivalently written in the form of $\lambda_{k+1}w_{k+1} + x_{k+1} - \tilde{v}_k = 0$, where $w_{k+1} \in \{\nabla\Phi(x_{k+1})\}$. This motivates us to introduce a relative error tolerance~\citep{Solodov-1999-Approximate,Monteiro-2010-Complexity}. In particular, we define the $\varepsilon$-subdifferential of a function $f$ by 
\begin{equation}\label{Def:SD}
\partial_\epsilon f(x) := \{w \in \br^d \mid f(y) \geq f(x) + \langle y - x, w\rangle - \epsilon, \ \forall y \in \br^d\}, 
\end{equation}
and find $\lambda_{k+1} > 0$ and a triple $\left(x_{k+1}, w_{k+1}, \varepsilon_{k+1}\right)$ such that $\|\lambda_{k+1}w_{k+1} + x_{k+1} - \tilde{v}_k\|^2 + 2\lambda_{k+1}\epsilon_{k+1} \leq  \sigma^2\|x_{k+1} - \tilde{v}_k\|^2$, where $w_{k+1} \in \partial_{\epsilon_{k+1}} \Phi(x_{k+1})$. To this end, $w_{k+1}$ is a sufficiently accurate approximation of $\nabla \Phi(x_{k+1})$. Moreover, the second equation can be relaxed to $\lambda_{k+1}\|x_{k+1} - \tilde{v}_k\|^{p-1} \geq \theta$.  
\begin{remark}
We present our first conceptual algorithmic framework formally in Algorithm~\ref{Algorithm:CAF-I}. This scheme includes the \textsf{large-step A-HPE} framework~\citep{Monteiro-2013-Accelerated} as a special instance. Indeed, it reduces to the \textsf{large-step A-HPE} framework if we set $y = \tilde{y}$ and $p = 2$ and change the notation of $(x, v, \tilde{v}, w)$ to $(y, x, \tilde{x}, v)$ in~\citet{Monteiro-2013-Accelerated}.
\end{remark}
\begin{algorithm}[!t]
\begin{algorithmic}\caption{Conceptual Algorithmic Framework I}\label{Algorithm:CAF-I}
\STATE \textbf{STEP 0:}  Let $x_0, v_0 \in \br^d$, $\sigma \in (0, 1)$ and $\theta > 0$ be given, and set $A_0 = 0$ and $k = 0$. 
\STATE \textbf{STEP 1:} If $0 = \nabla\Phi(x_k)$, then \textbf{stop}. 
\STATE \textbf{STEP 2:} Otherwise, compute $\lambda_{k+1} > 0$ and a triple $(x_{k+1}, w_{k+1}, \epsilon_{k+1}) \in \br^d \times \br^d \times (0, +\infty)$ such that
\begin{align*}
& w_{k+1} \in \partial_{\epsilon_{k+1}} \Phi(x_{k+1}), \\
& \|\lambda_{k+1}w_{k+1} + x_{k+1} - \tilde{v}_k\|^2 + 2\lambda_{k+1}\epsilon_{k+1} \leq \sigma^2\|x_{k+1} - \tilde{v}_k\|^2, \\
& \lambda_{k+1}\|x_{k+1} - \tilde{v}_k\|^{p-1} \geq \theta.
\end{align*}
where $\tilde{v}_k = \frac{A_k}{A_k + a_{k+1}}x_k + \frac{a_{k+1}}{A_k + a_{k+1}}v_k$ and $a_{k+1}^2 = \lambda_{k+1}(A_k + a_{k+1})$. 
\STATE \textbf{STEP 3:} Compute $A_{k+1} = A_k + a_{k+1}$ and $v_{k+1} = v_k - a_{k+1}w_{k+1}$. 
\STATE \textbf{STEP 4:}  Set $k \leftarrow k+1$, and go to \textbf{STEP 1}. 
\end{algorithmic}
\end{algorithm}
\paragraph{Second algorithmic framework.} By the definition of $\gamma(t)$, we have $(\frac{\dot{\gamma}(t)}{\gamma(t)})^2 = \lambda(t)\gamma(t)$ and $\gamma(0)=1$. This implies an equivalent formulation of the first-order system in Eq.~\eqref{Sys:FO-inverse} with $c = 2$: 
\begin{align*}
& \left\{\begin{array}{ll}
& \dot{v}(t) - \frac{\dot{\gamma}(t)}{\gamma^2(t)}\nabla \Phi(x(t)) = 0 \\
& \dot{x}(t) - \frac{\dot{\gamma}(t)}{\gamma(t)}(x(t) - v(t)) +  \frac{(\dot{\gamma}(t))^2}{(\gamma(t))^3}\nabla\Phi(x(t)) = 0 \\ 
& \gamma(t) = 4(\int_0^t \sqrt{\lambda(s)} ds + c)^{-2} \\
& (\lambda(t))^p\|\nabla\Phi(x(t))\|^{p-1} = \theta \\
& (x(0), v(0)) = (x_0, v_0)
\end{array}\right. \\
\Longleftrightarrow & 
\left\{\begin{array}{ll}
& \dot{v}(t) + \frac{\alpha(t)}{\gamma(t)}\nabla \Phi(x(t)) = 0 \\ 
& \dot{x}(t) + \alpha(t)(x(t) - v(t)) + \lambda(t)\nabla\Phi(x(t)) = 0 \\
& (\alpha(t))^2 = \lambda(t)\gamma(t), \ \dot{\gamma}(t) + \alpha(t)\gamma(t) = 0 \\
& (\lambda(t))^p\|\nabla\Phi(x(t))\|^{p-1} = \theta \\ 
& (x(0), v(0), \gamma(0)) = (x_0, v_0, 1).  
\end{array}\right.
\end{align*}
We define discrete-time sequences, $\{(x_k, v_k, \lambda_k, \alpha_k, \gamma_k)\}_{k \geq 0}$, that correspondx to the continuous-time sequences $\{(x(t), v(t), \lambda(t), \alpha(t), \gamma(t))\}_{t \geq 0}$. From implicit time discretization, we have
\begin{equation}\label{DSys:FO-inverse}
\left\{\begin{array}{ll}
& v_{k+1} - v_k + \frac{\alpha_{k+1}}{\gamma_{k+1}}\nabla\Phi(x_{k+1}) = 0 \\ 
& x_{k+1} - x_k + \alpha_{k+1}(x_k - v_k) + \lambda_{k+1}\nabla\Phi(x_{k+1}) = 0 \\
& (\alpha_{k+1})^2 = \lambda_{k+1}\gamma_{k+1}, \ \gamma_{k+1} = (1-\alpha_{k+1})\gamma_k, \ \gamma_0 = 1 \\
& (\lambda_{k+1})^p\|\nabla\Phi(x_{k+1})\|^{p-1} = \theta.    
\end{array}\right.
\end{equation}
By introducing a new variable $\tilde{v}_k = (1 - \alpha_{k+1})x_k + \alpha_{k+1}v_k$, the second and fourth lines of Eq.~\eqref{DSys:FO} can be equivalently reformulated as  
\begin{equation*}
\lambda_{k+1}\nabla\Phi(x_{k+1}) + x_{k+1} - \tilde{v}_k = 0, \qquad \lambda_{k+1}\|x_{k+1} - \tilde{v}_k\|^{p-1} = \theta. 
\end{equation*}
By the same approximation strategy as before, we solve these two equations inexactly and replace $\nabla \Phi(x_{k+1})$ by a sufficiently accurate approximation in the first line of Eq.~\eqref{DSys:FO-inverse}.
\begin{remark}
We present our second conceptual algorithmic framework formally in Algorithm~\ref{Algorithm:CAF-II}. To the best of our knowledge, this scheme does not appear in the literature and is based on an estimate sequence which differs from the one used in Algorithm~\ref{Algorithm:CAF-I}. However, from a continuous-time perspective, these two algorithms are equivalent up to a constant $c>0$, demonstrating that they achieve the same convergence rate in terms of both objective function gap and squared gradient norm. 
\end{remark}
\begin{algorithm}[!t]
\begin{algorithmic}\caption{Conceptual Algorithmic Framework II}\label{Algorithm:CAF-II}
\STATE \textbf{STEP 0:}  Let $x_0, v_0 \in \br^d$, $\sigma \in (0, 1)$ and $\theta  > 0$ be given, and set $\gamma_0 = 1$ and $k = 0$. 
\STATE \textbf{STEP 1:} If $0 = \nabla\Phi(x_k)$, then \textbf{stop}. 
\STATE \textbf{STEP 2:} Otherwise, compute $\lambda_{k+1} > 0$ and a triple $(x_{k+1}, w_{k+1}, \epsilon_{k+1}) \in \br^d \times \br^d \times (0, +\infty)$ such that
\begin{align*}
& w_{k+1} \in \partial_{\epsilon_{k+1}} \Phi(x_{k+1}), \\ 
& \|\lambda_{k+1}w_{k+1} + x_{k+1} - \tilde{v}_k\|^2 + 2\lambda_{k+1}\epsilon_{k+1} \leq \sigma^2\|x_{k+1} - \tilde{v}_k\|^2, \\ 
& \lambda_{k+1}\|x_{k+1} - \tilde{v}_k\|^{p-1} \geq \theta. 
\end{align*}
where $\tilde{v}_k = (1 - \alpha_{k+1})x_k + \alpha_{k+1}v_k$ and $(\alpha_{k+1})^2 = \lambda_{k+1}(1-\alpha_{k+1})\gamma_k$. 
\STATE \textbf{STEP 3:} Compute $\gamma_{k+1} = (1 - \alpha_{k+1})\gamma_k$ and $v_{k+1} = v_k - \frac{\alpha_{k+1}}{\gamma_{k+1}}w_{k+1}$. 
\STATE \textbf{STEP 4:}  Set $k \leftarrow k+1$, and go to \textbf{STEP 1}. 
\end{algorithmic}
\end{algorithm} 
\paragraph{Comparison with G\"{u}ler's accelerated proximal point algorithm.} Algorithm~\ref{Algorithm:CAF-II} is related to G\"{u}ler’s accelerated proximal point algorithm (APPA)~\citep{Guler-1992-New}, which combines Nesterov acceleration~\citep{Nesterov-1983-Method} and Martinet's PPA~\citep{Martinet-1970-Regularisation,Martinet-1972-Determination}. Indeed, the analogs of update formulas $\tilde{v}_k = (1 - \alpha_{k+1})x_k + \alpha_{k+1}v_k$ and $(\alpha_{k+1})^2 = \lambda_{k+1}(1-\alpha_{k+1})\gamma_k$ appear in G\"{u}ler's algorithm, suggesting similar evolution dynamics. However, G\"{u}ler's APPA does not specify how to choose $\{\lambda_k\}_{k \geq 0}$ but regard them as the parameters, while our algorithm links its choice with the gradient norm of $\Phi$ via the large-step condition. 

Such difference is emphasized by recent studies on the continuous-time perspective of G\"{u}ler's APPA~\citep{Attouch-2019-Time,Attouch-2019-Fast}. More specifically,~\citet{Attouch-2019-Fast} proved that G\"{u}ler's APPA can be interpreted as the implicit time discretization of an open-loop inertial gradient system (see~\citet[Eq.~(53)]{Attouch-2019-Fast}): 
\begin{equation*}
\ddot{x}(t) + \left(g(t) - \frac{\dot{g}(t)}{g(t)}\right)\dot{x}(t) + \beta(t)\nabla\Phi(x(t)) = 0. 
\end{equation*}
where $g_k$ and $\beta_k$ in their notation correspond to $\alpha_k$ and $\lambda_k$ in Algorithm~\ref{Algorithm:CAF-II}. By using $\gamma_{k+1} - \gamma_k = -\alpha_{k+1}\gamma_k$ and standard continuous-time arguments, we have $g(t) = -\frac{\dot{\gamma}(t)}{\gamma(t)}$ and $\beta(t) = \lambda(t) = \frac{(\dot{\gamma}(t))^2}{(\gamma(t))^3}$. By further defining $a(t) = \frac{1}{\gamma(t)}$, the above system is in the form of 
\begin{equation}\label{Sys:Guler}
\ddot{x}(t) + \left(\frac{2\dot{a}(t)}{a(t)} - \frac{\ddot{a}(t)}{\dot{a}(t)}\right)\dot{x}(t) + \left(\frac{(\dot{a}(t))^2}{a(t)}\right)\nabla\Phi(x(t)) = 0,  
\end{equation}
where $a$ explicitly depends on the variable $\lambda$ as follows,
\begin{equation*}
\begin{array}{l}
a(t) = \frac{1}{4}(\int_0^t \sqrt{\lambda(s)} ds + 2)^2.   
\end{array} 
\end{equation*}
Compared to our closed-loop control system, the one in Eq.~\eqref{Sys:Guler} is open-loop without the algebra equation and does not contain Hessian-driven damping. The coefficient for the gradient term is also different, standing for different time rescaling in the evolution dynamics~\citep{Attouch-2021-Fast}.

\subsection{Complexity analysis}
We study the iteration complexity of Algorithm~\ref{Algorithm:CAF-I} and~\ref{Algorithm:CAF-II}. Our analysis is largely motivated by the aforementioned continuous-time analysis, simplifying the analysis in~\citet{Monteiro-2013-Accelerated} for the case of $p=2$ and generalizing it to the case of $p>2$ in a systematic manner (see Theorem~\ref{Theorem:CAFI-Main} and Theorem~\ref{Theorem:CAFII-Main}). Throughout this subsection, $x^\star$ denotes the projection of $v_0$ onto the solution set of $\Phi$.

\paragraph{Algorithm~\ref{Algorithm:CAF-I}.} We start with the presentation of our main results for Algorithm~\ref{Algorithm:CAF-I}, which in fact generalizes~\citet[Theorem~4.1]{Monteiro-2013-Accelerated} to the case of $p>2$. 
\begin{theorem}\label{Theorem:CAFI-Main}
For every integer $k \geq 1$, the objective function gap satisfies
\begin{equation*}
\Phi(x_k) - \Phi(x^\star) = O(k^{-\frac{3p+1}{2}}), 
\end{equation*}
and 
\begin{equation*}
\inf_{1 \leq i \leq k} \|w_i\|^2 = O(k^{-3p}), \quad \inf_{1 \leq i \leq k} \epsilon_i = O(k^{-\frac{3p+3}{2}}). 
\end{equation*}
\end{theorem}
Note that the only difference between Algorithm~\ref{Algorithm:CAF-I} and \textsf{large-step A-HPE} framework in~\citet{Monteiro-2013-Accelerated} is the order in the algebraic equation. As such, many of the technical results derived in~\cite{Monteiro-2013-Accelerated} also hold for Algorithm~\ref{Algorithm:CAF-I}; more specifically, ~\citet[Theorem~3.6, Lemma~3.7 and Proposition~3.9]{Monteiro-2013-Accelerated}. 

We also present a technical lemma that provides a lower bound for $A_k$. 
\begin{lemma}\label{Lemma:CAFI-Key}
For $p \geq 1$ and every integer $k \geq 1$, we have
\begin{equation*}
A_k \geq \left(\frac{\theta(1-\sigma^2)^{\frac{p-1}{2}}}{(p+1)^{\frac{3p+1}{2}}\|v_0-x^\star\|^{p-1}}\right) k^{\frac{3p+1}{2}}. 
\end{equation*}
\end{lemma}
\begin{proof}
For $p = 1$, the large-step condition implies that $\lambda_k \geq \theta$ for all $k \geq 0$. By~\citet[Lemma~3.7]{Monteiro-2013-Accelerated}, we have $A_k \geq \frac{\theta k^2}{4}$. 

For $p \geq 2$, the large-step condition implies that
\begin{align*}
\lefteqn{\sum_{i=1}^k A_i(\lambda_i)^{-\frac{p+1}{p-1}}\theta^{\frac{2}{p-1}} \leq \sum_{i=1}^k A_i(\lambda_i)^{-\frac{p+1}{p-1}}(\lambda_i\|x_i - \tilde{v}_{i-1}\|^{p-1})^{\frac{2}{p-1}}} \\
& = \sum_{i=1}^k \frac{A_i}{\lambda_i}\|x_i - \tilde{v}_{i-1}\|^2 \overset{\textnormal{~\citet[Theorem~3.6]{Monteiro-2013-Accelerated}}}{\leq} \frac{\|v_0 - x^\star\|^2}{1-\sigma^2}.   
\end{align*}
By the H\"{o}lder inequality, we have
\begin{equation*}
\sum_{i=1}^k (A_i)^{\frac{p-1}{3p+1}} = \sum_{i=1}^k (A_i(\lambda_i)^{-\frac{p+1}{p-1}})^{\frac{p-1}{3p+1}} (\lambda_i)^{\frac{p+1}{3p+1}} \leq (\sum_{i=1}^k A_i(\lambda_i)^{-\frac{p+1}{p-1}})^{\frac{p-1}{3p+1}}(\sum_{i=1}^k \sqrt{\lambda_i})^{\frac{2p+2}{3p+1}}.
\end{equation*}
For the ease of presentation, we define $C = \theta^{-\frac{2}{3p+1}}(\frac{\|v_0 - x^\star\|^2}{1-\sigma^2})^{\frac{p-1}{3p+1}}$. Putting these pieces together yields: 
\begin{equation}\label{inequality:CAFI-Key-first}
\sum_{i=1}^k (A_i)^{\frac{p-1}{3p+1}} \leq C\left(\sum_{i=1}^k \sqrt{\lambda_i}\right)^{\frac{2p+2}{3p+1}} \overset{\textnormal{~\citet[Lemma~3.7]{Monteiro-2013-Accelerated}}}{\leq} 2C(A_k)^{\frac{p+1}{3p+1}}.  
\end{equation}
The remaining proof is based on the Bihari-LaSalle inequality in discrete time. In particular, we define $\{y_k\}_{k \geq 0}$ by $y_k = \sum_{i=1}^k (A_i)^{\frac{p-1}{3p+1}}$. Then, $y_0=0$ and Eq.~\eqref{inequality:CAFI-Key-first} implies that 
\begin{equation*}
y_k \leq 2C(y_k - y_{k-1})^{\frac{p+1}{p-1}}. 
\end{equation*}
This implies that 
\begin{equation}\label{inequality:CAFI-Key-second}
y_k - y_{k-1} \geq \left(\frac{y_k}{2C}\right)^{\frac{p-1}{p+1}} \Longrightarrow \frac{y_k - y_{k-1}}{(y_k)^{\frac{p-1}{p+1}}} \geq \left(\frac{1}{2C}\right)^{\frac{p-1}{p+1}}. 
\end{equation}
Inspired by the continuous-time inequality in Lemma~\ref{Lemma:CAFI-Key}, we claim that the following discrete-time inequality holds for every integer $k \geq 1$:
\begin{equation}\label{inequality:CAFI-Key-third}
(y_k)^{\frac{2}{p+1}} - (y_{k-1})^{\frac{2}{p+1}} \geq \frac{2}{p+1}\left(\frac{y_k - y_{k-1}}{[y_k]^{\frac{p-1}{p+1}}}\right). 
\end{equation}
Indeed, we define $g(t) = 1 - t^{\frac{2}{p+1}}$ and find that this function is convex for $\forall t \in (0, 1)$ since $p \geq 1$. Thus, we have
\begin{equation*}
1 - t^{\frac{2}{p+1}} = g(t) - g(1) \geq (t-1)\nabla g(1) = \frac{2(1-t)}{p+1} \Longrightarrow \frac{1 - t^{\frac{2}{p+1}}}{1-t} \geq \frac{2}{p+1}. 
\end{equation*}
Since $y_k$ is increasing, we have $\frac{y_{k-1}}{y_k} \in (0, 1)$. Then, the desired Eq.~\eqref{inequality:CAFI-Key-second} follows from setting $t=\frac{y_{k-1}}{y_k}$. Combining Eq.~\eqref{inequality:CAFI-Key-second} and Eq.~\eqref{inequality:CAFI-Key-third} yields that 
\begin{equation*}
(y_k)^{\frac{2}{p+1}} - (y_{k-1})^{\frac{2}{p+1}} \geq \frac{2}{p+1}\left(\frac{1}{2C}\right)^{\frac{p-1}{p+1}}. 
\end{equation*}
Therefore, we conclude that 
\begin{equation*}
(y_k)^{\frac{2}{p+1}} = (y_0)^{\frac{2}{p+1}} + \left(\sum_{i=1}^k (y_i)^{\frac{2}{p+1}} - (y_{i-1})^{\frac{2}{p+1}}\right) \geq \frac{2}{p+1}\left(\frac{1}{2C}\right)^{\frac{p-1}{p+1}}k. 
\end{equation*}
By the definition of $y_k$, we have
\begin{equation*}
\sum_{i=1}^k (A_i)^{\frac{p-1}{3p+1}} \geq \left(\frac{2}{p+1}\right)^{\frac{p+1}{2}}\left(\frac{1}{2C}\right)^{\frac{p-1}{2}} k^{\frac{p+1}{2}}. 
\end{equation*}
This together with Eq.~\eqref{inequality:CAFI-Key-first} yields that  
\begin{equation*}
A_k \geq \left(\frac{1}{2C}\sum_{i=1}^k (A_i)^{\frac{p-1}{3p+1}}\right)^{\frac{3p+1}{p+1}} \geq \left(\frac{1}{(p+1)C}\right)^{\frac{3p+1}{2}} k^{\frac{3p+1}{2}}.   
\end{equation*}
This completes the proof. 
\end{proof}
\begin{remark}
The proof of Lemma~\ref{Lemma:CAFI-Key} is much simpler than the existing analysis; e.g.,~\citet[Lemma~4.2]{Monteiro-2013-Accelerated} for the case of $p=2$ and~\citet[Theorem~3.4]{Jiang-2019-Optimal} and~\citet[Lemma~3.3]{Bubeck-2019-Near} for the case of $p \geq 2$. Notably, it is not a generalization of the highly technical proof in~\citet[Lemma~4.2]{Monteiro-2013-Accelerated} but can be interpreted as the discrete-time counterpart of the proof of Lemma~\ref{Lemma:Lyapunov-Key}. 
\end{remark}
\paragraph{Proof of Theorem~\ref{Theorem:CAFI-Main}:} For every integer $k \geq 1$, by~\citet[Theorem~3.6]{Monteiro-2013-Accelerated} and Lemma~\ref{Lemma:CAFI-Key}, we have
\begin{equation*}
\Phi(x_k) - \Phi(x^\star) \leq \frac{\|v_0-x^\star\|^2}{2A_k} = O(k^{-\frac{3p+1}{2}}). 
\end{equation*}
Combining~\citet[Proposition~3.9]{Monteiro-2013-Accelerated} and Lemma~\ref{Lemma:CAFI-Key}, we have
\begin{align*}
\inf_{1 \leq i \leq k} \lambda_i\|w_i\|^2 & \leq \frac{1+\sigma}{1-\sigma}\frac{\|v_0-x^\star\|^2}{\sum_{i=1}^k A_i} = O(k^{-\frac{3p+3}{2}}), \\ 
\inf_{1 \leq i \leq k} \varepsilon_i & \leq \frac{\sigma^2}{2(1-\sigma^2)}\frac{\|v_0-x^\star\|^2}{\sum_{i=1}^k A_i} = O(k^{-\frac{3p+3}{2}}). 
\end{align*}
In addition, we have $\|\lambda_i w_i + x_i - \tilde{v}_{i-1}\| \leq \sigma\|x_i - \tilde{v}_{i-1}\|$ and $\lambda_i\|x_i - \tilde{v}_{i-1}\|^{p-1} \geq \theta$. This implies that $\lambda_i\|w_i\|^{\frac{p-1}{p}} \geq \theta^{\frac{1}{p}}(1-\sigma)^{\frac{p-1}{p}}$. Putting these pieces together yields that $\inf_{1 \leq i \leq k} \|w_i\|^{\frac{p+1}{p}} = O(k^{-\frac{3p+3}{2}})$ which implies that
\begin{equation*}
\inf_{1 \leq i \leq k} \|w_i\|^2 = \left(\inf_{1 \leq i \leq k} \|w_i\|^{\frac{p+1}{p}}\right)^{\frac{2p}{p+1}} = O(k^{-3p}). 
\end{equation*}
This completes the proof. 

\paragraph{Algorithm~\ref{Algorithm:CAF-II}.} We now present our main results for Algorithm~\ref{Algorithm:CAF-II}. The proof is analogous to that of Theorem~\ref{Theorem:CAFI-Main} and based on another estimate sequence. 
\begin{theorem}\label{Theorem:CAFII-Main}
For every integer $k \geq 1$, the objective function gap satisfies
\begin{equation*}
\Phi(x_k) - \Phi(x^\star) = O(k^{-\frac{3p+1}{2}}) 
\end{equation*}
and 
\begin{equation*}
\inf_{1 \leq i \leq k} \|w_i\|^2 = O(k^{-3p}), \quad \inf_{1 \leq i \leq k} \epsilon_i = O(k^{-\frac{3p+3}{2}}). 
\end{equation*}
\end{theorem}
Inspired by the continuous-time Lyapunov function in Eq.~\eqref{Def:Lyapunov}, we construct a discrete-time Lypanunov function for Algorithm~\ref{Algorithm:CAF-II} as follows: 
\begin{equation}\label{Def:Lyapunov-discrete}
\ECal_k = \frac{1}{\gamma_k}(\Phi(x_k) - \Phi(x^\star)) + \frac{1}{2}\|v_k - x^\star\|^2. 
\end{equation}
We use this function to prove technical results that pertain to Algorithm~\ref{Algorithm:CAF-II} and which are the analogs of~\citet[Theorem~3.6, Lemma~3.7 and Proposition~3.9]{Monteiro-2013-Accelerated}.
\begin{lemma}\label{Lemma:CAFII-Basic-first}
For every integer $k \geq 1$, 
\begin{equation*}
\frac{1-\sigma^2}{2}\left(\sum_{i=1}^k \frac{1}{\lambda_i\gamma_i}\|x_i - \tilde{v}_{i-1}\|^2\right) \leq \ECal_0 - \ECal_k,
\end{equation*}
which implies that 
\begin{equation*}
\Phi(x_k) - \Phi(x^\star) \leq \gamma_k\ECal_0, \quad \|v_k - x^\star\| \leq \sqrt{2\ECal_0}. 
\end{equation*}
Assuming that $\sigma < 1$, we have $\sum_{i=1}^k \frac{1}{\lambda_i\gamma_i} \|x_i - \tilde{v}_{i-1}\|^2 \leq \frac{2\ECal_0}{1-\sigma^2}$. 
\end{lemma}
\begin{proof}
It suffices to prove the first inequality which implies the other results. Based on the discrete-time Lyapunov function, we define two functions $\phi_k: \br^d \mapsto \br$ and $\Gamma_k: \br^d \mapsto \br$ by ($\Gamma_k$ is related to $\ECal_k$ and defined recursively):
\begin{align*}
& \phi_k(v) = \Phi(x_k) + \langle v - x_k, w_k\rangle - \epsilon_k - \Phi(x^\star), \ \forall k \geq 0, \\
& \Gamma_0(v) = \frac{1}{\gamma_0}(\Phi(x_0) - \Phi(x^\star)) + \frac{1}{2}\|v - v_0\|^2, \ \Gamma_{k+1} = \Gamma_k + \frac{\alpha_{k+1}}{\gamma_{k+1}}\phi_{k+1}, \ \forall k \geq 0.
\end{align*}
First, by definition, $\phi_k$ is affine. Since $w_{k+1} \in \partial_{\epsilon_{k+1}} \Phi(x_{k+1})$, Eq.~\eqref{Def:SD} implies that $\phi_k(v) \leq \Phi(v) - \Phi(x^\star)$. Furthermore, $\Gamma_k$ is quadratic and $\nabla^2 \Gamma_k = \nabla^2\Gamma_0$ since $\phi_k$ is affine. Then, we prove that $\Gamma_k(v) \leq \Gamma_0(v) + \frac{1 - \gamma_k}{\gamma_k}(\Phi(v) - \Phi(x^\star))$ using induction. Indeed, it holds when $k=0$ since $\gamma_0=1$. Assuming that this inequality holds for $\forall i \leq k$, we derive from $\phi_k(v) \leq \Phi(v) - \Phi(x^\star)$ and $\gamma_{k+1} = (1 - \alpha_{k+1})\gamma_k$ that 
\begin{equation*}
\Gamma_{k+1}(v) \leq \Gamma_0(v) + \left(\frac{1 - \gamma_k}{\gamma_k}+\frac{\alpha_{k+1}}{\gamma_{k+1}}\right)(\Phi(v) - \Phi(x^\star)) = \Gamma_0(v) + \frac{1 - \gamma_k}{\gamma_k}(\Phi(v) - \Phi(x^\star)).   
\end{equation*}
Finally, we prove that $v_k = \argmin_{v \in \br^d} \Gamma_k(v)$ using the induction. Indeed, it holds when $k=0$. Suppose that this inequality holds for $\forall i \leq k$, we have
\begin{equation*}
\nabla \Gamma_{k+1}(v) = \nabla\Gamma_k(v) + \frac{\alpha_{k+1}}{\gamma_{k+1}}\nabla\phi_{k+1}(v) = v - v_k + \frac{\alpha_{k+1}}{\gamma_{k+1}}w_{k+1}. 
\end{equation*}
Using the definition of $v_k$ and the fact that $\gamma_{k+1} = (1 - \alpha_{k+1})\gamma_k$, we have $\nabla \Gamma_{k+1}(v) = 0$ if and only if $v = v_{k+1}$.

The remaining proof is based on the gap sequence $\{\beta_k\}_{k \geq 0}$ which is defined by $\beta_k = \inf_{v \in \br^d} \Gamma_k(v) - \frac{1}{\gamma_k}(\Phi(x_k) - \Phi(x^\star))$. Using the previous facts that $\Gamma_k$ is quadratic with $\nabla^2 \Gamma_k=1$ and the upper bound for $\Gamma_k(v)$, we have  
\begin{equation*}
\beta_k = \Gamma_k(x^\star) - \frac{1}{\gamma_k}(\Phi(x_k) - \Phi(x^\star)) - \frac{1}{2}\|x^\star - v_k\|^2 \leq \Gamma_0(x^\star) - \ECal_k = \ECal_0 - \ECal_k.  
\end{equation*}
By definition, we have $\beta_0 = 0$. Thus, it suffices to prove that the following recursive inequality holds true for every integer $k \geq 0$, 
\begin{equation}\label{inequality:CAFII-recursive-gap}
\beta_{k+1} \geq \beta_k + \frac{1-\sigma^2}{2\lambda_{k+1}\gamma_{k+1}}\|x_{k+1} - \tilde{v}_k\|^2. 
\end{equation}
In particular, we define $\tilde{v} = (1 - \alpha_{k+1})x_k + \alpha_{k+1}v$ for any given $v \in \br^d$. Using the definition of $\tilde{v}_k$ and the affinity of $\phi_{k+1}$, we have
\begin{align}
\phi_{k+1}(\tilde{v}) & = (1 - \alpha_{k+1})\phi_{k+1}(x_k) + \alpha_{k+1}\phi_{k+1}(v), \label{CAFII:rgap-first} \\ 
\tilde{v} - \tilde{v}_k & = \alpha_{k+1}(v - v_k). \label{CAFII:rgap-second}
\end{align}
Since $\Gamma_k$ is quadratic with $\nabla^2 \Gamma_k=1$, we have $\Gamma_k(v) = \Gamma_k(v_k) + \frac{1}{2}\|v - v_k\|^2$. Plugging this into the recursive equation for $\Gamma_k$ yields that 
\begin{equation*}
\Gamma_{k+1}(v) = \Gamma_k(v_k) + \frac{1}{2}\|v - v_k\|^2 + \frac{\alpha_{k+1}}{\gamma_{k+1}}\phi_{k+1}(v). 
\end{equation*}
By the definition of $\beta_k$, we have $\Gamma_k(v_k) = \beta_k + \frac{1}{\gamma_k}(\Phi(x_k) - \Phi(x^\star))$. Putting these pieces together with the definition of $\ECal_k$ yields that 
\begin{equation*}
\Gamma_{k+1}(v) = \beta_k + \frac{\alpha_{k+1}}{\gamma_{k+1}}\phi_{k+1}(v) + \frac{1}{\gamma_k}(\Phi(x_k) - \Phi(x^\star)) + \frac{1}{2}\|v - v_k\|^2. 
\end{equation*}
Since $\phi_{k+1}(v) \leq \Phi(v) - \Phi(x^\star)$, we have
\begin{align*}
\lefteqn{\Gamma_{k+1}(v) \geq \beta_k + \frac{\alpha_{k+1}}{\gamma_{k+1}}\phi_{k+1}(v) + \frac{1}{\gamma_k}\phi_{k+1}(x_k) + \frac{1}{2}\|v - v_k\|^2} \\
& \overset{\textnormal{Eq.~\eqref{CAFII:rgap-first}}}{=} \beta_k + \frac{1}{\gamma_{k+1}}\phi_{k+1}(\tilde{v}) + \frac{1}{2}\|v - v_k\|^2 \\
& \quad = \beta_k + \frac{1}{\gamma_{k+1}}\left(\phi_{k+1}(\tilde{v}) + \frac{\gamma_{k+1}}{2}\|v - v_k\|^2\right) \\
& \overset{\textnormal{Eq.~\eqref{CAFII:rgap-second}}}{=} \beta_k + \frac{1}{\gamma_{k+1}}\left(\phi_{k+1}(\tilde{v}) + \frac{\gamma_{k+1}}{2(\alpha_{k+1})^2}\|\tilde{v} - \tilde{v}_k\|^2\right) \\
& \quad = \beta_k + \frac{1}{\gamma_{k+1}}\left(\phi_{k+1}(\tilde{v}) + \frac{1}{2\lambda_{k+1}}\|\tilde{v} - \tilde{v}_k\|^2\right). 
\end{align*}
Using~\citet[Lemma~3.3]{Monteiro-2013-Accelerated} with $\lambda = \lambda_{k+1}$, $\tilde{v} = \tilde{v}_k$, $\tilde{x} = x_{k+1}$, $\tilde{w} = w_{k+1}$ and $\epsilon = \epsilon_{k+1}$, we have
\begin{equation*}
\inf_{v \in \br^d} \left\{\langle v - x_{k+1}, w_{k+1}\rangle - \epsilon_{k+1} + \frac{1}{2\lambda_{k+1}}\|v - \tilde{v}_k\|^2\right\} \geq \frac{1-\sigma^2}{2\lambda_{k+1}}\|x_{k+1} - \tilde{v}_k\|^2. 
\end{equation*}
which implies that 
\begin{equation*}
\phi_{k+1}(\tilde{v}) + \frac{1}{2\lambda_{k+1}}\|\tilde{v} - \tilde{v}_k\|^2 - \frac{1}{\gamma_{k+1}}(\Phi(x_{k+1}) - \Phi(x^\star)) \geq \frac{1-\sigma^2}{2\lambda_{k+1}}\|x_{k+1} - \tilde{v}_k\|^2.
\end{equation*}
Putting these pieces together yields that  
\begin{equation*}
\inf_{v \in \br^d} \Gamma_{k+1}(v) - \frac{1}{\gamma_{k+1}}(\Phi(x_{k+1}) - \Phi(x^\star)) \geq \beta_k + \frac{1-\sigma^2}{2\lambda_{k+1}\gamma_{k+1}}\|x_{k+1} - \tilde{v}_k\|^2. 
\end{equation*}
which together with the definition of $\beta_k$ yields the desired inequality in Eq.~\eqref{inequality:CAFII-recursive-gap}. This completes the proof. 
\end{proof}
\begin{lemma}\label{Lemma:CAFII-Basic-second}
For every integer $k \geq 0$, it holds that 
\begin{equation*}
\sqrt{\frac{1}{\gamma_{k+1}}} \geq \sqrt{\frac{1}{\gamma_k}} + \frac{1}{2}\sqrt{\lambda_{k+1}}. 
\end{equation*}
As a consequence, the following statements hold: (i) For every integer $k \geq 0$, it holds that $\gamma_k \leq (1+\frac{1}{2}\sum_{j=1}^k \sqrt{\lambda_j})^{-2}$; (ii) If $\sigma < 1$ is further assumed, then we have $\sum_{j=1}^k \|x_j - \tilde{v}_{j-1}\|^2 \leq \frac{2\ECal_0}{1-\sigma^2}$.
\end{lemma}
\begin{proof}
It suffices to prove the first inequality which implies the other results. By the definition of $\{\gamma_k\}_{k \geq 0}$ and $\{\alpha_k\}_{k \geq 0}$, we have $\gamma_{k+1}=(1-\alpha_{k+1})\gamma_k$ and $(\alpha_{k+1})^2=\lambda_{k+1}\gamma_{k+1}$. This implies that 
\begin{equation*}
\frac{1}{\gamma_k} = \frac{1}{\gamma_{k+1}} - \frac{\alpha_{k+1}}{\gamma_{k+1}} = \frac{1}{\gamma_{k+1}} - \sqrt{\frac{\lambda_{k+1}}{\gamma_{k+1}}}. 
\end{equation*}
Since $\gamma_k > 0$ and $\lambda_k > 0$, we have $\sqrt{\frac{1}{\gamma_{k+1}}} \geq \frac{1}{2}\sqrt{\lambda_{k+1}}$ and  
\begin{equation*}
\frac{1}{\gamma_k} \leq \frac{1}{\gamma_{k+1}} - \sqrt{\frac{\lambda_{k+1}}{\gamma_{k+1}}} + \frac{\lambda_{k+1}}{4} = \left(\sqrt{\frac{1}{\gamma_{k+1}}} - \frac{1}{2}\sqrt{\lambda_{k+1}}\right)^2. 
\end{equation*} 
which implies the desired inequality.
\end{proof}
\begin{lemma}\label{Lemma:CAFII-Basic-third}
For every integer $k \geq 1$ and $\sigma<1$, there exists $1 \leq i \leq k$ such that 
\begin{equation*}
\inf_{1 \leq i \leq k} \sqrt{\lambda_i}\|w_i\| \leq \sqrt{\frac{1+\sigma}{1-\sigma}}\sqrt{\frac{2\ECal_0}{\sum_{i=1}^k \frac{1}{\gamma_i}}}, \quad \inf_{1 \leq i \leq k} \epsilon_i \leq \frac{\sigma^2}{2(1-\sigma^2)}\frac{2\ECal_0}{\sum_{i=1}^k \frac{1}{\gamma_i}}. 
\end{equation*}
\end{lemma}
\begin{proof}
With the convention $0/0=0$, we define $\tau_k = \max\{\frac{2\epsilon_k}{\sigma^2}, \frac{\lambda_k\|w_k\|^2}{(1+\sigma)^2}\}$ for every integer $k \geq 1$. Then, we have
\begin{align*}
& 2\lambda_k\epsilon_k \leq \sigma^2\|x_k - \tilde{v}_{k-1}\|^2, \\
& \|\lambda_k w_k\| \leq \|\lambda_k w_k + x_k - \tilde{v}_{k-1}\| + \|x_k - \tilde{v}_{k-1}\| \leq (1+\sigma)\|x_k - \tilde{v}_{k-1}\|. 
\end{align*}
which implies that $\lambda_k\tau_k \leq \|x_k - \tilde{v}_{k-1}\|^2$ for every integer $k \geq 1$. This together with Lemma~\ref{Lemma:CAFII-Basic-first} yields that 
\begin{equation*}
\frac{2\ECal_0}{1-\sigma^2} \geq \sum_{i=1}^k \frac{1}{\lambda_i\gamma_i}\|x_i - \tilde{v}_{i-1}\|^2 \geq \left(\inf_{1 \leq i \leq k} \tau_i\right)\left(\sum_{i=1}^k \frac{1}{\gamma_i}\right). 
\end{equation*}
Combining this inequality with the definition of $\tau_k$ yields the desired results. 
\end{proof}
As the analog of Lemma~\ref{Lemma:CAFI-Key}, we provide a technical lemma on the upper bound for $\gamma_k$. The analysis is based on the same idea for proving Lemma~\ref{Lemma:CAFI-Key} and is motivated by continuous-time analysis for the first-order system in Eq.~\eqref{Sys:FO-inverse}.
\begin{lemma}\label{Lemma:CAFII-Key}
For $p \geq 1$ and every integer $k \geq 1$, we have
\begin{equation*}
\gamma_k \leq \frac{(p+1)^{\frac{3p+1}{2}}}{\theta}\left(\frac{2\ECal_0}{1-\sigma^2}\right)^{\frac{p-1}{2}} k^{-\frac{3p+1}{2}}.
\end{equation*}
\end{lemma}
\begin{proof}
For $p = 1$, the large-step condition implies that $\lambda_k \geq \theta$ for all $k \geq 0$. By Lemma~\ref{Lemma:CAFII-Basic-second}, we have $\gamma_k \leq \frac{4}{\theta k^2}$. 

For $p \geq 2$, the large-step condition implies that 
\begin{align*}
\lefteqn{\sum_{i=1}^k (\gamma_i)^{-1}(\lambda_i)^{-\frac{p+1}{p-1}}\theta^{\frac{2}{p-1}} \leq \sum_{i=1}^k (\gamma_i)^{-1}(\lambda_i)^{-\frac{p+1}{p-1}}(\lambda_i\|x_i - \tilde{v}_{i-1}\|^{p-1})^{\frac{2}{p-1}}} \\ 
& = \sum_{i=1}^k \frac{1}{\lambda_i\gamma_i}\|x_i - \tilde{v}_{i-1}\|^2 \overset{\text{Lemma}~\ref{Lemma:CAFII-Basic-first}}{\leq} \frac{2\ECal_0}{1-\sigma^2}. 
\end{align*}
By the H\"{o}lder inequality, we have
\begin{equation*}
\sum_{i=1}^k (\gamma_i)^{-\frac{p-1}{3p+1}} = \sum_{i=1}^k \left(\frac{1}{(\lambda_i)^{\frac{p+1}{p-1}}\gamma_i}\right)^{\frac{p-1}{3p+1}} (\lambda_i)^{\frac{p+1}{3p+1}} \leq \left(\sum_{i=1}^k \frac{1}{(\lambda_i)^{\frac{p+1}{p-1}}\gamma_i}\right)^{\frac{p-1}{3p+1}}\left(\sum_{i=1}^k \sqrt{\lambda_i}\right)^{\frac{2p+2}{3p+1}}.
\end{equation*}
For ease of presentation, we define $C = \theta^{-\frac{2}{3p+1}}(\frac{2\ECal_0}{1-\sigma^2})^{\frac{p-1}{3p+1}}$. Putting these pieces together yields that 
\begin{equation}\label{inequality:CAFII-Key-first}
\sum_{i=1}^k (\gamma_i)^{-\frac{p-1}{3p+1}} \leq C\left(\sum_{i=1}^k \sqrt{\lambda_i}\right)^{\frac{2p+2}{3p+1}} \overset{\text{Lemma}~\ref{Lemma:CAFII-Basic-second}}{\leq} 2C(\gamma_k)^{-\frac{p+1}{3p+1}}.  
\end{equation}
Using the same argument for proving Lemma~\ref{Lemma:CAFI-Key}, we have
\begin{equation*}
\sum_{i=1}^k (\gamma_i)^{-\frac{p-1}{3p+1}} \geq \left(\frac{2}{p+1}\right)^{\frac{p+1}{2}}\left(\frac{1}{2C}\right)^{\frac{p-1}{2}} k^{\frac{p+1}{2}}. 
\end{equation*}
This together with Eq.~\eqref{inequality:CAFII-Key-first} yields that  
\begin{equation*}
\frac{1}{\gamma_k} \geq \left(\frac{1}{2C}\sum_{i=1}^k (\gamma_i)^{-\frac{p-1}{3p+1}}\right)^{\frac{3p+1}{p+1}} \geq \left(\frac{1}{(p+1)C}\right)^{\frac{3p+1}{2}} k^{\frac{3p+1}{2}}.   
\end{equation*}
This completes the proof. 
\end{proof}
\paragraph{Proof of Theorem~\ref{Theorem:CAFII-Main}:} For every integer $k \geq 1$, by Lemma~\ref{Lemma:CAFII-Basic-first} and Lemma~\ref{Lemma:CAFII-Key}, we have
\begin{equation*}
\Phi(x_k) - \Phi(x^\star) \leq \gamma_k\ECal_0 = O(k^{-\frac{3p+1}{2}}). 
\end{equation*}
By Lemma~\ref{Lemma:CAFII-Basic-third} and Lemma~\ref{Lemma:CAFII-Key}, we have
\begin{align*}
& \inf_{1 \leq i \leq k} \lambda_i\|w_i\|^2 \leq \frac{1+\sigma}{1-\sigma}\frac{2\ECal_0}{\sum_{i=1}^k \frac{1}{\gamma_i}} = O(k^{-\frac{3p+3}{2}}), \\ 
& \inf_{1 \leq i \leq k} \epsilon_i \leq \frac{\sigma^2}{2(1-\sigma^2)}\frac{2\ECal_0}{\sum_{i=1}^k \gamma_i} = O(k^{-\frac{3p+3}{2}}). 
\end{align*}
As in the proof of Theorem~\ref{Theorem:CAFI-Main}, we conclude that $\inf_{1 \leq i \leq k} \|w_i\|^2 = O(k^{-3p})$. This completes the proof. 
\begin{remark}
The discrete-time analysis in this subsection is based on a discrete-time Lyapunov function in Eq.~\eqref{Def:Lyapunov-discrete}, which is closely related to the continuous one in Eq.~\eqref{Def:Lyapunov}, and two simple yet nontrivial technical lemmas (see Lemma~\ref{Lemma:CAFI-Key} and~\ref{Lemma:CAFII-Key}), which are both discrete-time versions of Lemma~\ref{Lemma:Lyapunov-Key}. Notably, the proofs of Lemma~\ref{Lemma:CAFI-Key} and~\ref{Lemma:CAFII-Key} follows the same path for proving Lemma~\ref{Lemma:Lyapunov-Key} and have demanded the use of the Bihari-LaSalle inequality in discrete time.
\end{remark}

\subsection{Optimal tensor algorithms and gradient norm minimization}
By instantiating Algorithm~\ref{Algorithm:CAF-I} and~\ref{Algorithm:CAF-II} with approximate tensor subroutines, we develop two families of optimal $p$-th order tensor algorithms for minimizing the function $\Phi \in \FCal_\ell^p(\br^d)$. The former one include all of existing optimal $p$-th order tensor algorithms~\citep{Gasnikov-2019-Optimal,Jiang-2019-Optimal,Bubeck-2019-Near} while the latter one is new to our knowledge. We also provide one hitherto unknown result that the optimal $p$-th order tensor algorithms in this section minimize the squared gradient norm at a rate of $O(k^{-3p})$. The results extend those for the optimal first-order and second-order algorithms that have been obtained in~\citet{Shi-2018-Understanding} and~\citet{Monteiro-2013-Accelerated}. 

\paragraph{Approximate tensor subroutine.}  The celebrated proximal point algorithms~\citep{Rockafellar-1976-Monotone,Guler-1992-New} (corresponding to implicit time discretization of certain systems) require solving an exact proximal iteration with proximal coefficient $\lambda > 0$ at each iteration:   
\begin{equation}\label{subprob:prox-exact}
x = \argmin_{u \in \br^d} \ \left\{\Phi(u) + \frac{1}{2\lambda}\left\|u - v\right\|^2\right\}. 
\end{equation}
In general, Eq.~\eqref{subprob:prox-exact} can be as hard as minimizing the function $\Phi$ when the proximal coefficient $\lambda \rightarrow +\infty$. Fortunately, when $\Phi \in \FCal_\ell^p(\br^d)$, it suffices to solve the subproblem that minimizes the sum of the $p$-th order Taylor approximation of $\Phi$ and a regularization term, motivating a line of $p$-th order tensor algorithms~\citep{Baes-2009-Estimate,Birgin-2016-Evaluation,Birgin-2017-Worst,Martinez-2017-High,Nesterov-2019-Implementable,Jiang-2020-Unified,Gasnikov-2019-Optimal,Jiang-2019-Optimal,Bubeck-2019-Near}. More specifically, we define   
\begin{equation*}
\Phi_v(u) = \Phi(v) + \langle\nabla\Phi(v), u-v\rangle + \sum_{j=2}^p \frac{1}{j!}\nabla^{(j)}\Phi(v)[u-v]^j + \frac{\ell\|u - v\|^{p+1}}{(p+1)!}.   
\end{equation*}
The algorithms of this subsection are based on either an inexact solution of Eq.~\eqref{subprob:inexact-first}, used in~\citet{Jiang-2019-Optimal}, or an exact solution of Eq.~\eqref{subprob:inexact-second}, used in~\citet{Gasnikov-2019-Optimal} and~\citet{Bubeck-2019-Near}: 
\begin{subequations}
\begin{equation}\label{subprob:inexact-first}
\min_{u \in \br^d} \ \Phi_v(u) + \frac{1}{2\lambda}\|u - v\|^2,
\end{equation}
\begin{equation}\label{subprob:inexact-second}
\min_{u \in \br^d} \ \Phi_v(u). 
\end{equation}
\end{subequations}
In particular, the solution $x_v$ of Eq.~\eqref{subprob:inexact-first} is unique and satisfies $\lambda\nabla\Phi_v(x_v) + x_v - v = 0$. Thus, we denote a \textit{$\hat{\sigma}$-inexact solution} of Eq.~\eqref{subprob:inexact-first} by a vector $x \in \br^d$ satisfying that $\|\lambda\nabla\Phi_v(x) + x - v\| \leq \hat{\sigma}\|x - v\|$ use either it or an exact solution of Eq.~\eqref{subprob:inexact-second} in our tensor algorithms. 
\begin{algorithm}[!t]
\begin{algorithmic}\caption{Optimal $p$-th order Tensor Algorithm I~\citep{Gasnikov-2019-Optimal,Jiang-2019-Optimal,Bubeck-2019-Near}}\label{Algorithm:Optimal-I}
\STATE \textbf{STEP 0:}  Let $x_0, v_0 \in \br^d$, $\hat{\sigma} \in (0, 1)$ and $0 < \sigma_l < \sigma_u < 1$ such that $\sigma_l(1+\hat{\sigma})^{p-1} < \sigma_u(1-\hat{\sigma})^{p-1}$ and $\sigma = \hat{\sigma} + \sigma_u < 1$ be given, and set $A_0 = 0$ and $k=0$. 
\STATE \textbf{STEP 1:} If $0 = \nabla\Phi(x_k)$, then \textbf{stop}. 
\STATE \textbf{STEP 2:} Otherwise, compute a positive scalar $\lambda_{k+1}$ with a $\hat{\sigma}$-inexact solution $x_{k+1} \in \br^d$ of Eq.~\eqref{subprob:inexact-first} satisfying that
\begin{equation*}
\frac{\sigma_l p!}{2\ell} \leq \lambda_{k+1}\left\|x_{k+1} - \tilde{v}_k\right\|^{p-1} \leq \frac{\sigma_u p!}{2\ell}, 
\end{equation*}
or an exact solution $x_{k+1} \in \br^d$ of Eq.~\eqref{subprob:inexact-second} satisfying that
\begin{equation*}
\frac{(p-1)!}{2\ell} \leq \lambda_{k+1}\left\|x_{k+1} - \tilde{v}_k\right\|^{p-1} \leq \frac{p!}{\ell(p+1)}, 
\end{equation*}
where $\tilde{v}_k = \frac{A_k}{A_k + a_{k+1}}x_k + \frac{a_{k+1}}{A_k + a_{k+1}}v_k$ and $a_{k+1}^2 = \lambda_{k+1}(A_k + a_{k+1})$.
\STATE \textbf{STEP 3:} Compute $A_{k+1} = A_k + a_{k+1}$ and $v_{k+1} = v_k - a_{k+1}\nabla\Phi(x_{k+1})$. 
\STATE \textbf{STEP 4:}  Set $k \leftarrow k+1$, and go to \textbf{STEP 1}. 
\end{algorithmic}
\end{algorithm}
\paragraph{First algorithm.} We present the first optimal $p$-th order tensor algorithm in Algorithm~\ref{Algorithm:Optimal-I} and prove that it is Algorithm~\ref{Algorithm:CAF-I} with specific choice of $\theta$. 
\begin{proposition}\label{Prop:opt-I}
Algorithm~\ref{Algorithm:Optimal-I} is Algorithm~\ref{Algorithm:CAF-I} with $\theta = \frac{\sigma_l p!}{2\ell}$ or $\theta = \frac{(p-1)!}{2\ell}$. 
\end{proposition}
\begin{proof}
Given that a pair $(x_k, v_k)_{k \geq 1}$ is generated by Algorithm~\ref{Algorithm:Optimal-I}, we define $w_k = \nabla\Phi(x_k)$ and $\varepsilon_k = 0$. Then $v_{k+1} = v_k - a_{k+1}\nabla\Phi(x_{k+1}) = v_k - a_{k+1}w_{k+1}$. Using~\citet[Proposition~3.2]{Jiang-2019-Optimal} with a $\hat{\sigma}$-inexact solution $x_{k+1} \in \br^d$ of Eq.~\eqref{subprob:inexact-first} at $(\lambda_{k+1}, \tilde{v}_k)$, a triple $\left(x_{k+1}, w_{k+1}, \varepsilon_{k+1}\right) \in \br^d \times \br^d \times (0, +\infty)$ satisfies that 
\begin{equation*}
w_{k+1} \in \partial_{\epsilon_{k+1}} \Phi(x_{k+1}), \quad \|\lambda_{k+1}w_{k+1} + x_{k+1} - \tilde{v}_k\|^2 + 2\lambda_{k+1}\epsilon_{k+1} \leq \sigma^2\|x_{k+1} - \tilde{v}_k\|^2.
\end{equation*}
Since $\theta = \frac{\sigma_l p!}{2\ell} \in (0, 1)$ and $\sigma = \hat{\sigma} + \sigma_u < 1$, we have
\begin{equation*}
\begin{array}{rcl}
\lambda_{k+1}\left\|x_{k+1} - \tilde{v}_k\right\|^{p-1} \leq \frac{\sigma_u p!}{2\ell} & \Longrightarrow & \hat{\sigma} + \frac{2\ell\lambda_{k+1}}{p!}\|x_{k+1}-\tilde{v}_k\|^{p-1} \leq \hat{\sigma} + \sigma_u = \sigma, \\
\lambda_{k+1}\left\|x_{k+1} - \tilde{v}_k\right\|^{p-1} \geq \frac{\sigma_l p!}{2\ell} & \Longrightarrow & \lambda_{k+1}\left\|x_{k+1} - \tilde{v}_k\right\|^{p-1} \geq \theta. 
\end{array}
\end{equation*}
Using the same argument with~\citet[Lemma~3.1]{Bubeck-2019-Near} instead of~\citet[Proposition~3.2]{Jiang-2019-Optimal} and an exact solution $x_{k+1} \in \br^d$ of Eq.~\eqref{subprob:inexact-second}, we obtain the same result with $\theta = \frac{(p-1)!}{2\ell}$. Putting these pieces together yields the desired conclusion. 
\end{proof}
In view of Proposition~\ref{Prop:opt-I}, the iteration complexity derived for Algorithm~\ref{Algorithm:CAF-I} hold for Algorithm~\ref{Algorithm:Optimal-I}. We summarize the results in the following theorem. 
\begin{theorem}\label{Theorem:Optimal-I-Main}
For every integer $k \geq 1$, the objective function gap satisfies
\begin{equation*}
\Phi(x_k) - \Phi(x^\star) = O(k^{-\frac{3p+1}{2}}), 
\end{equation*}
and the squared gradient norm satisfies
\begin{equation*}
\inf_{1 \leq i \leq k} \|\nabla\Phi(x_i)\|^2 = O(k^{-3p}). 
\end{equation*}
\end{theorem}
\begin{remark}
Theorem~\ref{Theorem:Optimal-I-Main} has been derived in~\citet[Theorem~6.4]{Monteiro-2013-Accelerated} for the special case of $p=2$, and a similar result for Nesterov's accelerated gradient descent (the special case of $p=1$) has also been derived in~\citet{Shi-2018-Understanding}. For $p \geq 3$ in general, the first inequality on the objective function gap has been derived independently in~\citet[Theorem~1]{Gasnikov-2019-Optimal},~\citet[Theorem~3.5]{Jiang-2019-Optimal} and~\citet[Theorem~1.1]{Bubeck-2019-Near}, while the second inequality on the squared gradient norm is new to our knowledge. 
\end{remark}
\begin{algorithm}[!t]
\begin{algorithmic}\caption{Optimal $p$-th order Tensor Algorithm II}\label{Algorithm:Optimal-II}
\STATE \textbf{STEP 0:}  Let $x_0, v_0 \in \br^d$, $\hat{\sigma} \in (0, 1)$ and $0 < \sigma_l < \sigma_u < 1$ such that $\sigma_l(1+\hat{\sigma})^{p-1} < \sigma_u(1-\hat{\sigma})^{p-1}$ and $\sigma = \hat{\sigma} + \sigma_u < 1$ be given, and set $\gamma_0 = 1$ and $k=0$. 
\STATE \textbf{STEP 1:} If $0 = \nabla\Phi(x_k)$, then \textbf{stop}. 
\STATE \textbf{STEP 2:} Otherwise, compute a positive scalar $\lambda_{k+1}$ with a $\hat{\sigma}$-inexact solution $x_{k+1} \in \br^d$ of Eq.~\eqref{subprob:inexact-first} satisfying that
\begin{equation*}
\frac{\sigma_l p!}{2\ell} \leq \lambda_{k+1}\|x_{k+1} - \tilde{v}_k\|^{p-1} \leq \frac{\sigma_u p!}{2\ell}, 
\end{equation*}
or an exact solution $x_{k+1} \in \br^d$ of Eq.~\eqref{subprob:inexact-second} satisfying that 
\begin{equation*}
\frac{(p-1)!}{2\ell} \leq \lambda_{k+1}\|x_{k+1} - \tilde{v}_k\|^{p-1} \leq \frac{p!}{\ell(p+1)}, 
\end{equation*}
where $\tilde{v}_k = (1 - \alpha_{k+1})x_k + \alpha_{k+1}v_k$ and $(\alpha_{k+1})^2 = \lambda_{k+1}(1-\alpha_{k+1})\gamma_k$.
\STATE \textbf{STEP 3:} Compute $\gamma_{k+1} = (1 - \alpha_{k+1})\gamma_k$ and $v_{k+1} = v_k - \frac{\alpha_{k+1}\nabla\Phi(x_{k+1})}{\gamma_{k+1}}$. 
\STATE \textbf{STEP 4:}  Set $k \leftarrow k+1$, and go to \textbf{STEP 1}. 
\end{algorithmic}
\end{algorithm}
\paragraph{Second algorithm.} We present the second optimal $p$-th order tensor algorithm in Algorithm~\ref{Algorithm:Optimal-II} which is Algorithm~\ref{Algorithm:CAF-II} with specific choice of $\theta$. The proof is omitted since it is the same as the aforementioned analysis for Algorithm~\ref{Algorithm:Optimal-I}.  
\begin{proposition}
Algorithm~\ref{Algorithm:Optimal-II} is Algorithm~\ref{Algorithm:CAF-II} with $\theta = \frac{\sigma_l p!}{2\ell}$ or $\theta = \frac{(p-1)!}{2\ell}$. 
\end{proposition}
\begin{theorem}\label{Theorem:Optimal-II-Main}
For every integer $k \geq 1$, the objective gap satisfies
\begin{equation*}
\Phi(x_k) - \Phi(x^\star) = O(k^{-\frac{3p+1}{2}}),
\end{equation*}
and the squared gradient norm satisfies
\begin{equation*}
\inf_{1 \leq i \leq k} \|\nabla\Phi(x_i)\|^2 = O(k^{-3p}). 
\end{equation*}
\end{theorem}
\begin{remark}
The approximate tensor subroutine in Algorithm~\ref{Algorithm:Optimal-I} and~\ref{Algorithm:Optimal-II} can be efficiently implemented usinga novel bisection search scheme. We refer the interested readers to~\citet{Jiang-2019-Optimal} and~\citet{Bubeck-2019-Near} for the details. 
\end{remark}

\section{Conclusions}\label{sec:conclusions}
We have presented a closed-loop control system for modeling optimal tensor algorithms for smooth convex optimization and provided continuous-time and discrete-time Lyapunov functions for analyzing the convergence properties of this system and its discretization. Our framework provides a systematic way to derive discrete-time $p$-th order optimal tensor algorithms, for $p \geq 2$, and simplify existing analyses via the use of a Lyapunov function. A key ingredient in our framework is the algebraic equation, which is not present in the setting of $p=1$, but is essential for deriving optimal acceleration methods for $p \geq 2$. Our framework allows us to infer that a certain class of $p$-th order tensor algorithms minimize the squared norm of the gradient at a fast rate of $O(k^{-3p})$ for smooth convex functions.

It is worth noting that one could also consider closed-loop feedback control of the velocity. This is called nonlinear damping in the PDE literature; see~\citet{Attouch-2020-Fast} for recent progress in this direction. There are also several other avenues for future research. In particular, it is of interest to bring our perspective into register with the Lagrangian and Hamiltonian frameworks that have proved productive in recent work~\citep{Wibisono-2016-Variational,Diakonikolas-2020-Generalized,Muehlebach-2020-Optimization,Francca-2021-Dissipative} and the control-theoretic viewpoint of~\citet{Lessard-2016-Analysis} and~\citet{Hu-2017-Dissipativity}. We would hope for this study to provide additional insight into the geometric or dynamical role played by the algebraic equation for modeling the continuous-time dynamics. Moreover, we wish to study possible extensions of our framework to nonsmooth optimization by using differential inclusions~\cite{Vassilis-2018-Differential} and monotone inclusions. The idea is to consider the setting in which $0 \in T(x)$ where $T$ is a maximally monotone operator in a Hilbert space~\citep{Alvarez-2001-Inertial,Attouch-2011-Continuous,Mainge-2013-First,Attouch-2013-Global,Abbas-2014-Newton,Attouch-2016-Dynamic,Bot-2016-Second,Attouch-2018-Convergence,Attouch-2020-Convergence,Attouch-2020-Newton,Attouch-2020-Continuous}. Finally, given that we know that direct discretization of our closed-loop control system cannot recover Nesterov's optimal high-order tensor algorithms~\citep[Section~4.3]{Nesterov-2018-Lectures}, it is of interest to investigate the continuous-time limit of Nesterov's algorithms and see whether the algebraic equation plays a role in their analysis.

\section*{Acknowledgments}
The authors would like to thank associate editor and two anonymous reviewers for constructive comments that improve the presentation of this paper. This work was supported in part by the Mathematical Data Science program of the Office of Naval Research under grant number N00014-18-1-2764.

\bibliographystyle{plainnat}
\bibliography{ref}

\end{document}